\documentclass[a4paper,10pt]{article}

\usepackage[normalem]{ulem}

\usepackage{a4wide}
\usepackage{amsmath,amssymb,amsfonts,amsthm}
\usepackage{mathrsfs}
\usepackage{mathtools}
\usepackage[font=small,labelfont=footnotesize]{caption}
\usepackage{subcaption}
\usepackage{enumitem}
\setlist{nolistsep}
\usepackage{multirow,tabularx}
\usepackage{makecell}
\usepackage{units}
\usepackage{authblk}
\usepackage{comment}
\usepackage[dvipsnames]{xcolor}
\usepackage{float}
\usepackage{setspace}
\usepackage{framed}

\usepackage{tikz}
\usepackage{pgfplots}
\pgfplotsset{every axis/.append style={
		label style={font=\scriptsize},
		tick label style={font=\scriptsize},
		legend style={font=\tiny\sf, row sep=-2pt},
		legend cell align={left}
}}

\newtheorem{theorem}{Theorem}[section]
\newtheorem{lemma}[theorem]{Lemma}
\newtheorem{corollary}[theorem]{Corollary}
\newtheorem{proposition}[theorem]{Proposition}
\newtheorem{problem}[theorem]{Problem}

\newtheorem{assumption}[theorem]{Assumption}

\theoremstyle{remark}
\newtheorem{remark}[theorem]{Remark}

\usepackage[numbers,sort&compress]{natbib}

\usepgfplotslibrary{external}
\tikzexternalenable

\numberwithin{equation}{section}
\mathtoolsset{showonlyrefs}

\usepackage[mathlines]{lineno}
\usepackage{etoolbox}
\newcommand*\linenomathpatchAMS[1]{%
	\expandafter\pretocmd\csname #1\endcsname {\linenomathAMS}{}{}%
	\expandafter\pretocmd\csname #1*\endcsname{\linenomathAMS}{}{}%
	\expandafter\apptocmd\csname end#1\endcsname {\endlinenomath}{}{}%
	\expandafter\apptocmd\csname end#1*\endcsname{\endlinenomath}{}{}%
}
\expandafter\ifx\linenomath\linenomathWithnumbers
\let\linenomathAMS\linenomathWithnumbers
\patchcmd\linenomathAMS{\advance\postdisplaypenalty\linenopenalty}{}{}{}
\else
\let\linenomathAMS\linenomathNonumbers
\fi
\linenomathpatchAMS{gather}
\linenomathpatchAMS{multline}
\linenomathpatchAMS{align}
\linenomathpatchAMS{alignat}
\linenomathpatchAMS{flalign}

\usepackage{xcolor}
\usepackage[colorlinks=true,citecolor=red!75!green,linkcolor=blue!75!green]{hyperref}
\usepackage{colortbl}
\definecolor{Rhodamine}{rgb}{0.81, 0.44, 0.69}
\definecolor{ForestGreen}{rgb}{0.13,0.55,0.13}

\definecolor{HDG_SYM}{rgb}{1, 0, 0} 
\definecolor{CHDG_SYM}{rgb}{0, 0, 1} 
\definecolor{CHDG_SYM2}{rgb}{0.81, 0.44, 0.69} 
\definecolor{CHDG_UPW}{rgb}{0.13,0.55,0.13} 

\newcommand{\vecalg}[1]{\boldsymbol{\mathsf{#1}}}
\newcommand{\matalg}[1]{\boldsymbol{\mathsf{#1}}}

\newcommand{\im}{\imath}
\renewcommand{\i}{\im}

\newcommand{\grad}{\boldsymbol \nabla}
\renewcommand{\div}{\grad \cdot}

\newcommand{\BV}{\mathbf{V}}

\newcommand{\bn}{\mathbf{n}}

\newcommand{\bu}{\mathbf{u}}
\newcommand{\bv}{\mathbf{v}}

\newcommand{\bx}{\mathbf{x}}

\newcommand{\bzero}{\mathbf{0}}

\newcommand{\TD}{\mathrm{D}}

\newcommand{\TI}{\mathrm{I}}

\newcommand{\TN}{\mathrm{N}}

\newcommand{\TR}{\mathrm{R}}
\newcommand{\TS}{\mathrm{S}}

\title{A hybridizable discontinuous Galerkin method\\with transmission variables for time-harmonic\\ acoustic problems in heterogeneous media\footnote{Published in \emph{Journal of Computational Physics} (doi: \href{https://doi.org/10.1016/j.jcp.2025.114009}{10.1016/j.jcp.2025.114009}). Distributed under \href{https://creativecommons.org/licenses/by/4.0/}{Creative Commons CC BY 4.0} license.}}

\author[1]{S. Pescuma}
\author[2]{G. Gabard}
\author[3]{T. Chaumont-Frelet}
\author[1]{A. Modave}

\affil[1]{\footnotesize POEMS, CNRS, Inria, ENSTA, Institut Polytechnique de Paris, 91120 Palaiseau, France,
\href{mailto:simone.pescuma@ensta.fr}{\texttt{simone.pescuma@ensta.fr}}\\ \href{mailto:axel.modave@ensta.fr}{\texttt{axel.modave@ensta.fr}}}
\affil[2]{Laboratoire d’Acoustique de l’Université du Mans (LAUM), IA-GS, CNRS, 72085 Le Mans, France
\href{mailto:gwenael.gabard@univ-lemans.fr}{\texttt{gwenael.gabard@univ-lemans.fr}}}
\affil[3]{\footnotesize Inria, Laboratoire Paul Painlevé, Université de Lille, 59655 Villeneuve-d'Ascq, France
\href{mailto:theophile.chaumont@inria.fr}{\texttt{theophile.chaumont@inria.fr}}}

\date{}

\begin{document}

\maketitle

\begin{abstract}
	\setlength{\parindent}{0mm}
	\noindent
	We consider the finite element solution of time-harmonic wave propagation problems in heterogeneous media with hybridizable discontinuous Galerkin (HDG) methods.
	In the case of homogeneous media, it has been observed that the iterative solution of the linear system can be accelerated by hybridizing with transmission variables instead of numerical traces, as performed in standard approaches.
	In this work, we extend the HDG method with transmission variables, which is called the CHDG method, to the heterogeneous case with piecewise constant physical coefficients.
	In particular, we consider formulations with standard upwind and general symmetric fluxes.
	The CHDG hybridized system can be written as a fixed-point problem, which can be solved with stationary iterative schemes for a class of symmetric fluxes.
	The standard HDG and CHDG methods are systematically studied with the different numerical fluxes by considering a series of 2D numerical benchmarks.
	The convergence of standard iterative schemes is always faster with the extended CHDG method than with the standard HDG methods, with upwind and scalar symmetric fluxes.
\end{abstract}

\setlength{\parindent}{0mm}
\setlength{\parskip}{2mm}
\setstretch{1.1}

\section{Introduction}

Discontinuous Galerkin (DG) methods are widely used for solving boundary value problems because of their ability to provide high-fidelity numerical solutions for complicated geometric and physical configurations, see e.g.~\cite{shi2018, lu2004, chung2013, li2014, carstensen2016}.
In the context of time-harmonic acoustic problems, they allow the use of high-order polynomial basis functions, which limits the dispersion error that occurs when considering high-frequency cases, e.g.~\cite{beriot2015, lieu2016, christodoulou2017, congreve2019}.
However, from a computational point of view, these methods require the solution of large sparse unstructured linear systems.
On the one hand, direct solvers require huge amounts of computation and are complicated to run efficiently on parallel computers.
On the other hand, iterative solvers require much less memory and allow efficient parallel implementations, but they may exhibit slow convergence due to intrinsic properties of the time-harmonic problems \cite{ernst2011difficult}.
To speed up convergence, preconditioning techniques and domain decomposition methods (DDM) have been and are being intensively studied, see e.g.~\cite{bayliss1983, erlangga2004, engquist2011, gander2019}.

Here, we consider the iterative solution of time-harmonic acoustic models with Hybridizable Discontinuous Galerkin (HDG) methods.
The principle of HDG methods is to introduce a hybrid variable on the mesh skeleton in order to decouple the physical variables at the interface between the elements.
Eliminating the physical variables within every element then leads to a hybridized system where the only unknown is the hybrid variable itself, see e.g.~\cite{barucq2021_2, barucq2023, chen2013, cockburn2008, giorgiani2013, nguyen2011, pham2024numerical}.
This approach modifies the size, structure, and conditioning of the global system that is actually solved, potentially allowing for more efficient solution procedures.
Once the hybrid variables are calculated, the physical variables inside each element can be recovered with a simple and inexpensive post-processing step.

In standard HDG methods, the so-called numerical trace is taken as the hybrid variable.
Here, we consider an alternative hybridization strategy where the hybrid variable
corresponds to a characteristic variable, which can be interpreted as a transmission variable or a Robin trace.
This approach, which we call CHDG, was recently introduced by two of the authors in~\cite{modave2023}. It
is related to some types of non-overlapping DDM \cite{collino2020, despres1991, farhat2009, modave2020}, Ultra Weak Variational Formulations (UWVF) \cite{barucq2021, barucq2024, cessenat1998, gabard2007, huttunen2002, parolin2022} and a hybridization technique studied in \cite{monk2010hybridizing}.

In Ref.~\cite{modave2023}, the performance of the DG, standard HDG and CHDG methods for fast iterative procedures has been compared in the cases of time-harmonic scalar wave propagation problems set in a homogeneous medium. In particular, it has been observed that the convergence of classical iterative schemes is faster with the CHDG method than with the standard DG and HDG methods.

In this work, we investigate extensions of the CHDG approach to scalar time-harmonic problems in heterogeneous media with material coefficients that are (possibly) discontinuous at the interfaces between the elements.
We consider DG schemes using either standard upwind fluxes, or a general family of symmetric fluxes based on differential operators defined on the interfaces between elements.
The corresponding transmission variables are used as hybrid variable in the hybridization process.
By extending the results in \cite{modave2023}, we show that the hybridized system is well-posed with both types of fluxes, and that it can be written in the form $(\TI-\Pi\TS)g=b$, with the set of all hybrid variables $g$, an exchange operator $\Pi$, a scattering operator $\TS$, and a global right-hand side $b$.
In the case of symmetric fluxes, we prove under reasonable assumptions that the operator $\Pi\TS$ is a strict contraction, and that the hybridized system can be solved with fixed-point iterations.

In practice, Krylov-type iterative solvers are widely used to solve time-harmonic problems.
Here, by using a set of academic benchmarks, we study and compare the performance of the GMRES iteration and the CGNR iteration (i.e.~the conjugate gradient applied to the normal equation) for solving the hybridized systems.
More specifically, we consider hybridized systems corresponding to the standard HDG and CHDG methods with different numerical fluxes.
We observe that, as in \cite{modave2023}, CHDG always requires fewer iterations than the standard HDG to reach a given accuracy with the GMRES and CGNR iterations.
Moreover, the convergence of standard upwind fluxes and symmetric fluxes is similar, and using symmetric fluxes with higher-order derivatives does not seem to speed up the convergence.

The remainder of this paper is structured as follows.
In Section \ref{problem}, we introduce the problem, the standard DG method, and the upwind and general symmetric numerical fluxes.
The standard HDG and CHDG methods are described and studied in Sections \ref{hybridization_HDG} and \ref{hybridization_CHDG}.
In particular, in the latter case, the strict contraction of the operator $\Pi\TS$ in the CHDG system is proven for a class of symmetric numerical fluxes.
In Section \ref{numerical_results}, the hybridizable methods and the numerical fluxes are systematically assessed by considering a set of 2D reference numerical benchmarks with different iterative schemes.
Conclusions and perspectives are proposed in Section \ref{conclusion}.

\section{Discontinuous Galerkin methods}
\label{problem}

Let $\Omega\subset\mathbb{R}^d$, with $d=2$ or $3$, be a Lipschitz polytopal domain.
The boundary $\partial\Omega$ of the domain is partitioned into three non-overlapping polytopal Lipschitz subsets $\Gamma_{\TD}$, $\Gamma_{\TN}$ and $\Gamma_{\TR}$, corresponding to the Dirichlet, Neumann and Robin boundaries, respectively.
We consider the following time-harmonic wave propagation problem:
\begin{align}
	\left\{
	\begin{aligned}
		-\im\kappa\eta^{-1}p + \div \bu & = 0,       &  & \text{in $\Omega$},     \\
		-\im\kappa\eta\bu + \grad p     & = \bzero,  &  & \text{in $\Omega$},     \\
		p                               & = s_{\TD}, &  & \text{on $\Gamma_\TD$}, \\
		\bn\cdot\bu                     & = s_{\TN}, &  & \text{on $\Gamma_\TN$}, \\
		p - \eta\bn\cdot\bu             & = s_{\TR}, &  & \text{on $\Gamma_\TR$},
	\end{aligned}
	\right.
	\label{eqn:pbm}
\end{align}
written for the unknown scalar field $p: \Omega \to \mathbb{C}$ and vector field $\bu: \Omega \to \mathbb{C}^d$ corresponding to the acoustic pressure and velocity.
On the boundary, we have introduced the source terms $s_{\TD}: \Gamma_{\TD} \to \mathbb{C}$, $s_{\TN}: \Gamma_{\TN} \to \mathbb{C}$ and $s_{\TR}: \Gamma_{\TR} \to \mathbb{C}$.
The propagation medium is characterised by its density $\rho:\Omega\rightarrow\mathbb{R}$ and sound speed $c:\Omega\rightarrow\mathbb{R}$.
These two quantities are strictly positive and assumed to be piecewise constant.
The problem \eqref{eqn:pbm} is written in the frequency domain, assuming a time dependence $e^{-\im\omega t}$ for the solution, where $\omega$ is the real-valued angular frequency.
We also introduced the wavenumber $\kappa:\Omega\rightarrow\mathbb{R}$ and the acoustic impedance $\eta:\Omega\rightarrow\mathbb{R}$ such that $\kappa=\omega/c$ and $\eta=\rho c$.
The vector field $\bn:\partial\Omega\rightarrow\mathbb{R}^d$ is the unit outward normal to $\Omega$.
For brevity, we do not consider source terms on the right-hand sides of the first two equations in \eqref{eqn:pbm}, but these could be included without difficulty.


\subsection{Mesh, approximation spaces and inner products}

To solve problem \eqref{eqn:pbm}, we use a conforming mesh $\mathcal{T}_h$ of the domain $\Omega$ consisting of simplicial elements $K$.
The collection of element boundaries is denoted by $\partial\mathcal{T}_h:=\{\partial K~|~K\in\mathcal{T}_h\}$, and the collection of faces is denoted by $\mathcal{F}_h$.
The collection of faces of an element $K$ is denoted by $\mathcal{F}_K$.

For simplicity, we fix a polynomial degree $\mathrm{p}\ge0$ and introduce the following approximation spaces for the scalar and vector fields:
\begin{align}
	V_h:=\prod_{K\in\mathcal{T}_h}\mathcal{P}_\mathrm{p}(K) \quad \text{and} \quad \BV_h:=\prod_{K\in\mathcal{T}_h}\boldsymbol{\mathcal{P}}_\mathrm{p}(K),
\end{align}
\noindent where $\mathcal{P}_\mathrm{p}(\cdot)$ and $\boldsymbol{\mathcal{P}}_\mathrm{p}(\cdot)$ denote spaces of scalar and vector complex-valued polynomial of degree smaller or equal to $\mathrm{p}$.
The restrictions of $u_h\in V_h$ and $\bu_h\in\BV_h$ on $K$ are denoted by $u_K$ and $\bu_K$, respectively. The coefficients $\kappa$ and $\eta$ are assumed to be constant on each element. Their values on $K$ are denoted by $\kappa_K$ and $\eta_K$, respectively.
For a face $F$ of $K$, $\bn_{K,F}$ is the unit normal on $F$ pointing outside $K$.

We introduce the following sesquilinear forms
\begin{align}
	(u,v)_K
	& :=\int_Ku\overline{v} \ d\bx,
	&
	(\bu,\bv)_K
	& :=\int_K\bu\cdot\overline{\bv} \ d\bx,
	&
	\langle u,v\rangle_{\partial K}
	& :=\sum_{F\in\mathcal{F}_K}\int_F u\overline{v} \ d\sigma(\bx),
	\\
	(u,v)_{\mathcal{T}_h}
	& :=\sum_{K\in\mathcal{T}_h}(u,v)_K,
	&
	(\bu,\bv)_{\mathcal{T}_h}
	& :=\sum_{K\in\mathcal{T}_h}(\bu,\bv)_K,
	&
	\langle u,v\rangle_{\partial\mathcal{T}_h}
	& :=\sum_{K\in\mathcal{T}_h}\langle u,v\rangle_{\partial K},
\end{align}
where the quantities used in the surface integral $\langle\cdot,\cdot\rangle_{\partial K}$ correspond to the restriction of fields or coefficients defined on $K$ (e.g.~$v_K$, $\bv_K$, $\kappa_K$ and $\eta_K$) or quantities associated with the faces of $K$ (e.g.~$\bn_{K,F}$ with $F\in\mathcal{F}_K$), unless explicitly specified.


\subsection{Variational formulation of the problem}

The general discontinuous Galerkin (DG) formulation of system \eqref{eqn:pbm} reads:
\begin{problem}
	\label{pbm:DGform}
	Find $(p_h,\bu_h)\in V_h\times\BV_h$ such that, for all $(q_h,\bv_h)\in V_h\times\BV_h$,
	\begin{align}
		\left\{
		\begin{aligned}
			- \i\big(\kappa\eta^{-1}p_h, q_h\big)_{\mathcal{T}_h}
			- \big(\bu_h, \grad q_h\big)_{\mathcal{T}_h}
			+ \big\langle\bn\cdot\widehat{\bu}, q_h\big\rangle_{\partial\mathcal{T}_h}
			& = 0, \\
			- \i\big(\kappa\eta\bu_h, \bv_h\big)_{\mathcal{T}_h}
			- \big(p_h, \grad\cdot\bv_h\big)_{\mathcal{T}_h}
			+ \big\langle\widehat{p}, \bn\cdot\bv_h\big\rangle_{\partial\mathcal{T}_h}
			& = 0,
		\end{aligned}
		\right.
	\end{align}
	with the \emph{numerical fluxes} $\widehat{p}$ and $\bn\cdot\widehat{\bu}$ to be defined.
\end{problem}

The numerical fluxes are defined face by face.
At each \emph{interior face} (i.e.~$F\not\subset\partial\Omega$), they depend on the values of the physical fields and coefficients associated to both neighboring elements.
At each \emph{boundary face} (i.e.~$F\subset\partial\Omega$), a specific definition is used to prescribe a boundary condition.
The numerical fluxes then depend on the boundary data.
Hereafter, two kinds of numerical fluxes are considered: standard upwind fluxes and general symmetric fluxes.

By convention, if $F$ is an interior face of an element $K$, then $K'$ is the neighboring element of $K$ sharing this face.
If $F$ is a boundary face, we define $\eta_{K'} := \eta_K$ to simplify the presentation, even if there is no neighboring element $K'$.


\subsection{Upwind numerical fluxes}

Standard upwind numerical fluxes are obtained by solving a Riemann problem associated to each interface between two elements, see e.g.~\cite{leveque2002, wilcox2010}.
For the wave propagation problem \eqref{eqn:pbm} with discontinuous coefficients between elements, the fluxes are given by Equations (9.58) at the end of Section 9.9 of Ref.~\cite{leveque2002}.
For a given face $F$, the \emph{upwind numerical fluxes} are defined as
\begin{align}
	\left\{
	\begin{aligned}
		\widehat{p}_F
		& := \frac{1}{\eta_K+\eta_{K'}} \left(\eta_{K'}g^{\oplus}_{K,F}+\eta_K g^{\ominus}_{K,F}\right),
		\\
		\bn_{K,F}\cdot\widehat{\bu}_F
		& := \frac{1}{\eta_K+\eta_{K'}} \left(g^{\oplus}_{K,F}- g^{\ominus}_{K,F}\right),
	\end{aligned}
	\right.
	\label{eqn:upwind:fluxes}
\end{align}
where $g^{\oplus}_{K,F}$ and $g^{\ominus}_{K,F}$ are the outgoing and incoming transmission variables, respectively.
In the standard terminology of Riemann solvers, the transmission variables are generally called characteristic variables, see e.g. \cite{toro2013, hesthaven2007}.
The \emph{outgoing transmission variable} depends on the physical variables and coefficients of the considered element $K$.
It is defined as
\begin{align}
	g^{\oplus}_{K,F}
	& := p_K+\eta_K\bn_{K,F}\cdot\bu_K.
	\label{eqn:upwind:outVar}
\end{align}
The \emph{incoming transmission variable} depends on whether there is a neighboring element $K'$ or not.
In the latter case, it depends on the boundary condition.
It is defined as
\begin{align}
	g^{\ominus}_{K,F}
	& :=
	\left\{
	\begin{aligned}
		& p_{K'}-\eta_{K'}\bn_{K,F}\cdot\bu_{K'} = g^{\oplus}_{K',F},
		&
		& \text{if } F\not\subset\partial\Omega,
		\\
		& 2s_{\TD}-g_{K,F}^{\oplus},
		&
		& \text{if } F\subset\Gamma_{\TD},
		\\
		& g_{K,F}^{\oplus}-2\eta_Ks_{\TN},
		&
		& \text{if } F\subset\Gamma_{\TN},
		\\
		& s_{\TR},
		&
		& \text{if } F\subset\Gamma_{\TR}.
	\end{aligned}
	\right.
	\label{eqn:upwind:incVar}
\end{align}
The upwind numerical fluxes can be written more explicitly as
\begin{align}
	& \left\{
	\begin{aligned}
		\widehat{p}_F
		& = \frac{\eta_K\eta_{K'}}{\eta_K+\eta_{K'}}\left(\frac{1}{\eta_K}p_K+\frac{1}{\eta_{K'}}p_{K'}+\bn_{K,F}\cdot\left(\bu_K-\bu_{K'}\right)\right),
		\\
		\bn_{K,F}\cdot\widehat{\bu}_F
		& = \frac{1}{\eta_K+\eta_{K'}}\left(\bn_{K,F}\cdot\left(\eta_K\bu_K+\eta_{K'}\bu_{K'}\right) + p_K-p_{K'}\right),
	\end{aligned}
	\right.
	&
	& \text{if $F\not\subset\partial\Omega$,}
	\\
	& \left\{
	\begin{aligned}
		\widehat{p}_F
		& = s_{\TD},
		\\
		\bn_{K,F}\cdot\widehat{\bu}_F
		& = \bn_{K,F}\cdot\bu_K + \frac{1}{\eta_K}\left(p_K - s_{\TD}\right),
	\end{aligned}
	\right.
	&
	& \text{if $F\subset\Gamma_\TD$,}
	\\
	& \left\{
	\begin{aligned}
		\widehat{p}_F
		& = p_K+\eta_{K}\left(\bn_{K,F}\cdot\bu_K-s_{\TN}\right),
		\\
		\bn_{K,F}\cdot\widehat{\bu}_F
		& = s_{\TN},
	\end{aligned}
	\right.
	&
	& \text{if $F\subset\Gamma_\TN$,}
	\\
	& \left\{
	\begin{aligned}
		\widehat{p}_F
		& = \frac{1}{2}\left(p_K+\eta_{K}\bn_{K,F}\cdot\bu_K+s_{\TR}\right),
		\\
		\bn_{K,F}\cdot\widehat{\bu}_F
		& = \frac{1}{2\eta_K}\left(p_K+\eta_{K}\bn_{K,F}\cdot\bu_K-s_{\TR}\right),
	\end{aligned}
	\right.
	&
	& \text{if $F\subset\Gamma_\TR$.}
\end{align}
These fluxes satisfy the relation
\begin{align}
	\widehat{p}_F + \eta_K \bn_{K,F}\cdot\widehat{\bu}_F
	= p_K + \eta_K \bn_{K,F}\cdot\bu_K,
	\label{eqn:upwind:relForHDG}
\end{align}
which will be useful to derive the standard HDG formulation.


\subsection{Symmetric numerical fluxes}

We now consider a general family of numerical fluxes where the coefficients do not depend on the elements, but only on the faces.
For each face $F$, we introduce a general operator $\mathcal{A}_F$ that verifies the following assumption where $\mathcal P_{\mathrm{p}}(F)$
denotes the space of polynomials of degree at most $\mathrm{p}$ defined on $F$.

\begin{assumption}
	\label{assumption}
	For each face $F$, the linear operator $\mathcal{A}_F: \mathcal P_{\mathrm{p}}(F) \to \mathcal P_{\mathrm{p}}(F)$ is positive and self-adjoint.
\end{assumption} 
For a given face $F$, the \emph{symmetric numerical fluxes} are defined as
\begin{align}
	\left\{
	\begin{aligned}
		\widehat{p}_F
		& := \frac{1}{2}\mathcal{A}_F \left(g^{\oplus}_{K,F}+g^{\ominus}_{K,F}\right),
		\\
		\bn_{K,F}\cdot\widehat{\bu}_F
		& := \frac{1}{2} \left(g^{\oplus}_{K,F}- g^{\ominus}_{K,F}\right),
	\end{aligned}
	\right.
	\label{eqn:sym:fluxes}
\end{align}
where the \emph{outgoing transmission variable} is defined as
\begin{align}
	g^{\oplus}_{K,F}
	& := \mathcal{A}_F^{-1}p_K+\bn_{K,F}\cdot\bu_K,
	\label{eqn:sym:outVar}
\end{align}
and the \emph{incoming transmission variable} is defined as
\begin{align}
	g^{\ominus}_{K,F}
	& :=
	\left\{
	\begin{aligned}
		& \mathcal{A}_F^{-1}p_{K'}-\bn_{K,F}\cdot\bu_{K'} = g_{K',F}^{\oplus},
		&
		& \text{if } F\not\subset\partial\Omega,
		\\
		& 2\mathcal{A}_F^{-1}s_\TD-g_{K,F}^{\oplus},
		&
		& \text{if } F\subset\Gamma_\TD,
		\\
		& g_{K,F}^{\oplus}-2 s_\TN,
		&
		& \text{if } F\subset\Gamma_\TN,
		\\
		& \mathcal{B}_{K,F,+}^{-1} \left(\mathcal{B}_{K,F,-} g_{K,F}^{\oplus} + 2\eta_K^{-1}s_\TR\right),
		&
		& \text{if } F\subset\Gamma_\TR,
	\end{aligned}
	\right.
	\label{eqn:sym:incVar}
\end{align}
with $\mathcal{B}_{K,F,\pm} := 1\pm\eta_K^{-1}\mathcal{A}_F$.
The symmetric numerical fluxes can be written more explicitly as
\begin{align}
	& \left\{
	\begin{aligned}
		\widehat{p}_F
		& = \frac{p_K+p_{K'}}{2}+\mathcal{A}_F\left(\bn_{K,F}\cdot\frac{\bu_K-\bu_{K'}}{2}\right),
		\\
		\bn_{K,F}\cdot\widehat{\bu}_F
		& = \bn_{K,F}\cdot\frac{\bu_K+\bu_{K'}}{2}+\mathcal{A}_F^{-1}\left(\frac{p_K-p_{K'}}{2}\right),
	\end{aligned}
	\right.
	&
	& \text{if $F\not\subset\partial\Omega$,}
	\\
	& \left\{
	\begin{aligned}
		\widehat{p}_F
		& = s_{\TD},
		\\
		\bn_{K,F}\cdot\widehat{\bu}_F
		& = \bn_{K,F}\cdot\bu_K + \mathcal{A}_F^{-1}\left(p_K - s_{\TD}\right),
	\end{aligned}
	\right.
	&
	& \text{if $F\subset\Gamma_\TD$,}
	\\
	& \left\{
	\begin{aligned}
		\widehat{p}_F
		& = p_K+\mathcal{A}_F\left(\bn_{K,F}\cdot\bu_K-s_{\TN}\right),
		\\
		\bn_{K,F}\cdot\widehat{\bu}_F
		& = s_{\TN},
	\end{aligned}
	\right.
	&
	& \text{if $F\subset\Gamma_\TN$,}
	\\
	& \left\{
	\begin{aligned}
		\widehat{p}_F
		& = \mathcal{B}_{K,F,+}^{-1}
		\left(p_K + \mathcal{A}_F\left(\bn_{K,F}\cdot\bu_K\right) + \eta_K^{-1}\mathcal{A}_F s_{\TR}\right),
		\\
		\bn_{K,F}\cdot\widehat{\bu}_F
		& = \eta_K^{-1} \mathcal{B}_{K,F,+}^{-1}
		\left(p_K + \mathcal{A}_F\left(\bn_{K,F}\cdot\bu_K\right) - s_{\TR}\right),
	\end{aligned}
	\right.
	&
	& \text{if $F\subset\Gamma_\TR$.}
\end{align}

Similar to the upwind fluxes, the following relation holds
\begin{align}
	\widehat{p}_F + \mathcal{A}_F\left(\bn_{K,F}\cdot\widehat{\bu}_F\right)
	& =  p_K + \mathcal{A}_F \left(\bn_{K,F}\cdot\bu_K\right).
	\label{eqn:sym:relForHDG}
\end{align}
It will be used in the derivation of the standard HDG formulation.


\subsubsection*{Zeroth-order symmetric fluxes}

If $\mathcal{A}_F=\mu_F$ is a strictly positive real parameter, the numerical fluxes simply read
\begin{align}
	\left\{
	\begin{aligned}
		\widehat{p}_F
		& := \frac{1}{2}\mu_F (g^{\oplus}_{K,F}+g^{\ominus}_{K,F}),
		\\
		\bn_{K,F}\cdot\widehat{\bu}_F
		& := \frac{1}{2} (g^{\oplus}_{K,F}-g^{\ominus}_{K,F}),
	\end{aligned}
	\right.
\end{align}
with
\begin{align}
	g^{\oplus}_{K,F}
	& := \frac{1}{\mu_F}p_K+\bn_{K,F}\cdot\bu_K,
	\\
	g^{\ominus}_{K,F}
	& := \frac{1}{\mu_F}p_{K'}-\bn_{K,F}\cdot\bu_{K'}, \qquad \text{if $F\not\subset\partial\Omega$}.
\end{align}
The definitions of the incoming transmission variables when $F\subset\partial\Omega$ are adapted accordingly.
Note that these fluxes are equivalent to the standard upwind fluxes if the impedance at the interfaces between elements is continuous, and $\mu_F=\eta_K=\eta_{K'}$ for neighboring elements $K$ and $K'$, regardless of whether the wavenumber is continuous or not.
We also point out that these transmission variables correspond to the lowest-order absorbing boundary
conditions approximating Dirichlet-to-Neumann (DtN) maps.

\subsubsection*{Second-order symmetric fluxes}

By analogy with non-overlapping domain decomposition methods (e.g.~\cite{modave2020}),
the operator $\mathcal{A}_F$ can be defined by using transmission operators based on more involved
domain truncation techniques, such as high-order absorbing boundary conditions.
We consider here an operator with second-order partial derivatives used in a second-order absorbing boundary conditions in \cite{despres2021}.
This operator is implicitly defined by
\begin{align}
	\langle \mathcal{A}_F \phi_F,v_F \rangle_F
	+
	\frac{1}{2\kappa_F^2} \langle \grad_F (\mathcal{A}_F \phi_F),\grad_F v_F\rangle_F
	=
	\mu_F \langle \phi_F,v_F\rangle_F, \quad \forall \phi_F,v_F \in \mathcal{P}_{\mathrm{p}}(F),
	\label{eqn:op2sym}
\end{align}
where the real parameter $\mu_F$ is strictly positive and $\grad_{\!F}$ is the gradient operator on $F$.
The operator $\mathcal{A}_F$ is then a discrete realization of
\begin{equation*}
	\mathcal{A}_F := \mu_F\left(1-\frac{1}{2\kappa_F^2}\Delta_F\right)^{-1},
\end{equation*}
where $\Delta_F$ is the Laplace--Beltrami operator on $F$,
completed by homogeneous Neumann boundary conditions at the extremities of $F$.
It satisfies Assumption~\ref{assumption}.

\begin{proposition} [Positivity and self-adjointness of the second-order operator]
	The operator $\mathcal{A}_F$ \eqref{eqn:op2sym} defined on an edge $F$ is positive and self-adjoint.
\end{proposition}

\begin{proof}
	Because we are in a finite dimensional framework, it is sufficient to prove that $\mathcal{A}_F^{-1}$ is positive and self-adjoint.
	First, we prove that $\mathcal{A}_F^{-1}$ is positive. To do so, given
	$\phi_F \in \mathcal P_{\mathrm{p}}(F)$, we let $u_F := \mathcal{A}_F\phi_F$, and
	pick $v_F = u_F$ in~\eqref{eqn:op2sym}. This gives
	\begin{align}
		0 \leq \langle u_F, u_F\rangle_F+\frac{1}{2\kappa_F^2}\langle\grad_F u_F,\grad_F u_F \rangle_F = \mu_F \langle \mathcal A_F^{-1} u_F,u_F \rangle_F,
	\end{align}
	which is the desired inequality, as $\mu_F > 0$.
	Then, we prove that $\mathcal{A}_F^{-1}$ is self-adjoint, which follows from the fact
	that the left-hand side in~\eqref{eqn:op2sym} is an Hermitian form. Indeed,
	\begin{align}
		\langle\mathcal{A}_F^{-1}u_F,v_F\rangle_F
		& =
		\frac{1}{\mu_F}\langle u_F,v_F \rangle_F
		+
		\frac{1}{2\kappa_F^2\mu_F} \langle \grad_F u_F,\grad_F v_F\rangle_F,
		\\
		& =
		\overline{
			\frac{1}{\mu_F}\langle v_F,u_F \rangle_F
			+
			\frac{1}{2\kappa_F^2\mu_F} \langle \grad_F v_F,\grad_F u_F\rangle_F,
		}
		\\
		& =
		\overline{\langle\mathcal{A}_F^{-1}v_F,u_F\rangle_F}
		=
		\langle u_F,\mathcal{A}_F^{-1}v_F\rangle_F
	\end{align}
	for all $u_F,v_F \in \mathcal{P}_{\mathrm{p}}(F)$.
\end{proof}

\section{Hybridization with numerical trace}
\label{hybridization_HDG}

In HDG methods, the hybridization consists in introducing an additional variable on the mesh skeleton to decouple the physical unknowns defined within neighboring elements.
The physical unknowns are then eliminated to obtain a hybridized system where the unknowns are associated only with the additional variable, called the \emph{hybrid variable}.
This step requires solving local element-wise problems.

We emphasize that this process does not affect the accuracy of the numerical solution, since it can be interpreted as an algebraic manipulation leading to an equivalent linear system.
When this linear system is solved with a direct solver, the physical unknowns that can be reconstructed are those of the original DG system.
However, the hybridized system has different algebraic properties that affect the iterative numerical solution.


\subsection{Standard HDG formulations}

In the standard HDG methods, the hybrid variable, denoted $\widehat{p}_h$, corresponds to the numerical flux $\widehat{p}$, frequently called the \emph{numerical trace} in the literature, see e.g. \cite{griesmaier2011}.
It belongs to the space of polynomials of order $\mathrm{p}$ on each face $F$ of the mesh, denoted by
\begin{align}
	\widehat{V}_h := \prod_{F\in\mathcal{F}_h}\mathcal{P}_\mathrm{p}(F).
\end{align}
For any field $\widehat{p}_h\in\widehat{V}_h$, there is one set of scalar unknowns associated to each face of the mesh.

For the case with symmetric numerical fluxes, the variable $\widehat{p}_h$ is defined by equation \eqref{eqn:sym:fluxes}.
Using equation \eqref{eqn:sym:relForHDG} in Problem \ref{pbm:DGform} then leads to the following HDG formulation.
\begin{problem}[HDG formulation with symmetric fluxes]
	\label{pbm:HDGform:sym}
	Find $(p_h,\bu_h,\widehat{p}_h)\in V_h\times\textbf{V}_h\times \widehat{V}_h$ such that, for all $(q_h,\bv_h,\widehat{q}_h)\in V_h\times\textbf{V}_h\times \widehat{V}_h$,
	\begin{align}
		\left\{
		\begin{aligned}
			- \i\big(\kappa\eta^{-1}p_h, q_h\big)_{\mathcal{T}_h}
			- \big(\bu_h, \grad q_h\big)_{\mathcal{T}_h}
			+ \big\langle \bn\cdot\bu_h + \mathcal{A}^{-1}(p_h-\widehat{p}_h), q_h \big\rangle_{\partial\mathcal{T}_h}
			& = 0, \\
			- \i\big(\kappa\eta\bu_h,\bv_h\big)_{\mathcal{T}_h}
			- \big(p_h,\grad\cdot\bv_h\big)_{\mathcal{T}_h}
			+ \big\langle\widehat{p}_h, \bn\cdot\bv_h\big\rangle_{\partial\mathcal{T}_h}
			& = 0,
		\end{aligned}
		\right.
	\end{align}
	and
	\begin{multline}
		\textstyle
		\big\langle \widehat{p}_h, \widehat{q}_h \big\rangle_{\mathcal{F}_h}
		- \big\langle \frac{1}{2}(p_h+\mathcal{A}(\bn\cdot\bu_h)), \widehat{q}_h \big\rangle_{\partial\mathcal{T}_h\setminus\partial\Omega}
		\\
		\textstyle
		- \big\langle p_h+\mathcal{A}(\bn\cdot\bu_h), \widehat{q}_h \big\rangle_{\Gamma_{\TN}}
		- \big\langle (1+\eta^{-1}\mathcal{A})^{-1} (p_h + \mathcal{A}(\bn\cdot\bu_h)), \widehat{q}_h \big\rangle_{\Gamma_{\TR}}
		\\
		\textstyle
		= \big\langle s_{\TD},\widehat{q}_h \big\rangle_{\Gamma_{\TD}}
		- \big\langle \mathcal{A} s_{\TN},\widehat{q}_h \big\rangle_{\Gamma_{\TN}}
		+ \big\langle (1+\eta^{-1}\mathcal{A})^{-1} (\eta^{-1}\mathcal{A} s_{\TR}),\widehat{q}_h \big\rangle_{\Gamma_{\TR}}.
	\end{multline}
\end{problem}

\begin{remark}[Case with upwind fluxes]
	For the case with the upwind numerical fluxes, the variable $\widehat{p}_h$ is defined by equation \eqref{eqn:upwind:fluxes}, and the HDG formulation is obtained similarly by using equation \eqref{eqn:upwind:relForHDG} in Problem \ref{pbm:DGform}.
	We refer to \cite{cockburn2009, cockburn2016} for the complete description of the HDG formulation with upwind fluxes.
\end{remark}


\subsection{Local element-wise discrete problems}

In the hybridization procedure, the physical fields $p_h$ and $\bu_h$ are eliminated by solving local element-wise problems, where the numerical trace $\widehat{p}_h$ is considered as a given datum.

For each element $K$, the local problem corresponding to the symmetric fluxes reads as follows.
\begin{problem}
	Find $(p_K,\bu_K)\in \mathcal{P}_\mathrm{p}(K)\times\boldsymbol{\mathcal{P}}_\mathrm{p}(K)$ such that, for all $(q_K,\bv_K)\in\mathcal{P}_\mathrm{p}(K)\times\boldsymbol{\mathcal{P}}_\mathrm{p}(K)$,
	\begin{align}
		\left\{
		\begin{aligned}
			&
			- \imath \big(\kappa_K\eta_K^{-1}p_K,q_K\big)_{K}
			- \big(\bu_K,\grad q_K\big)_{K}
			+ \sum_{F\in\mathcal{F}_K} \big\langle(\bn_{K,F}\cdot\bu_K + \mathcal{A}_F^{-1}p_K),q_K\big\rangle_{F}
			\\
			& \qquad\qquad
			= \sum_{F\in\mathcal{F}_K} \big\langle\mathcal{A}_F^{-1}\widehat{p}_F,q_K\big\rangle_{F},
			\\
			&
			- \imath \big(\kappa_K\eta_K\bu_K,\bv_K\big)_{K}
			- \big(p_K,\grad\cdot\bv_K\big)_{K}
			= - \sum_{F\in\mathcal{F}_K} \big\langle\widehat{p}_F,\bn_{K,F}\cdot\bv_K\big\rangle_{F},
		\end{aligned}
		\right.
	\end{align}
	for a given surface datum $\widehat{p}_F\in\mathcal{P}_\mathrm{p}(F)$ for all $F\in\mathcal{F}_K$.
	\label{loc_pbm:hdg}
\end{problem}
This local problem can be interpreted as a discretized Helmholtz problem defined on $K$ with a non-homogeneous Dirichlet boundary condition on $\partial K$.
The discrete problem is well-posed without any condition, as stated and proved below.
It is worth pointing out that at the continuous level, Helmholtz problems set in $K$ with Dirichlet
boundary conditions might not be well-posed, due to possible resonance frequencies.
Well-posedness at the discrete level follows from the fact that the solutions for the resonances do not belong to $\mathcal{P}_\mathrm{p}(K)$ and $\boldsymbol{\mathcal{P}}_\mathrm{p}(K)$, as it can be seen at the end of the proof.
\begin{theorem} [Well-posedness of the local discrete problem]
	Problem \ref{loc_pbm:hdg} is well-posed.
\end{theorem}
\begin{proof}
	We have to prove that, if $\widehat{p}_F=0$ for all $F\in\mathcal{F}_K$, the unique solution of Problem \ref{loc_pbm:hdg} is $p_K=0$ and $\textbf{u}_K=\textbf{0}$.
	For brevity, the subscripts $K$ and $F$ are omitted for the local fields, the test functions, the outgoing unit normal and the coefficients.
	Taking both equations of Problem \ref{loc_pbm:hdg} with $q=p$ and $\bv=\bu$ gives
	\begin{align}
		-\imath(\kappa\eta^{-1}p,p)_K
		-(\bu,\grad p)_K
		+\langle(\bn\cdot\bu + \mathcal{A}^{-1}p),p\rangle_{\partial K}
		& =0,
		\\
		-\imath(\kappa\eta\bu,\bu)_K-(p,\grad\cdot\bu)_K
		& =0.
	\end{align}
	Integrating by parts in both equations and taking the complex conjugate lead to
	\begin{align}
		\imath(\kappa\eta^{-1}p,p)_K
		+(p,\grad\cdot\bu)_K
		+\langle p,\mathcal{A}^{-1}p\rangle_{\partial K}
		& =0,
		\\
		\imath(\kappa\eta\bu,\bu)_K
		+(\bu,\grad p)_K
		-\langle\bn\cdot\bu,p\rangle_{\partial K}
		& =0.
	\end{align}
	Adding the four previous equations yields $\langle \mathcal{A}^{-1}p,p\rangle_{\partial K}=0$.
	Combined with Assumption \ref{assumption}, this implies that $p=0$ on $\partial K$.
	By using this result in Problem \ref{loc_pbm:hdg}, one has
	\begin{align}
		\left\{
		\begin{aligned}
			\displaystyle
			-\imath(\kappa\eta^{-1}p,q)_{K}+(\div\bu,q)_{K}
			& =0, \\
			-\imath(\kappa\eta\bu,\bv)_{K}+(\grad p,\bv)_{K}
			& =0,
		\end{aligned}
		\right.
	\end{align}
	for all $(q, \bv)\in\mathcal{P}_\mathrm{p}(K )\times\boldsymbol{\mathcal{P}}_\mathrm{p}(K)$.
	We conclude that
	\begin{align}
		-\imath\kappa\eta^{-1}p+\div\bu=0, \\
		-\imath\kappa\eta\bu+\grad p=0,
	\end{align}
	in a strong sense.
	Because there is no non-zero polynomial solution to the previous equations, this yields the result.
\end{proof}

\begin{remark}[Case with upwind fluxes]
	The well-posedness of the local problem in the case of upwind fluxes can be proved in a very similar way by following the same steps, see e.g. \cite{cockburn2009, cockburn2016}.
\end{remark}


\section{Hybridization with transmission variables}
\label{hybridization_CHDG}

Following the approach proposed in \cite{modave2023}, the hybridization is performed by taking the incoming transmission variable as hybrid variable.
As with standard hybridization, this approach does not alter the accuracy of the numerical solution when a direct solver is used, but it changes the properties of the linear system that needs to be solved.

\subsection{CHDG formulations}

An additional variable, denoted $g_h^{\ominus}$, corresponding to the incoming transmission variable is introduced at each face of each element.
This variable belongs to the space
\begin{align}
	G_h := \prod_{K\in\mathcal{T}_h}\prod_{F\in\mathcal{F}_K}\mathcal{P}_\mathrm{p}(F).
\end{align}
Each field of this space has one value associated to each face of each element.
Therefore, there are two values per interior face of the mesh, and one value per boundary face.

For the case with upwind fluxes, the hybrid variable is defined by equation \eqref{eqn:upwind:incVar}, and the numerical fluxes are defined by equation \eqref{eqn:sym:fluxes}.
For the case with symmetric fluxes, equations \eqref{eqn:sym:incVar} and \eqref{eqn:upwind:fluxes} are used instead.
The following formulation is obtained in both cases.
\begin{problem}[CHDG formulation with upwind/symmetric fluxes]
	Find $(p_h,\bu_h,g_h^{\ominus})\in V_h\times\textbf{V}_h\times G_h$ such that, for all $(q_h,\bv_h,\xi_h)\in V_h\times\textbf{V}_h\times G_h$,
	\begin{align}
		\left\{
		\begin{aligned}
			- \imath\big( \kappa\eta^{-1}p_h, q_h \big)_{\mathcal{T}_h}
			- \big(\bu_h, \grad q_h \big)_{\mathcal{T}_h}
			+ \big\langle \bn\cdot\widehat{\bu}(g^{\oplus}(p_h,\bu_h),g^{\ominus}_h), q_h \big\rangle_{\partial\mathcal{T}_h}
			& = 0, \\
			- \imath\big( \kappa\eta\bu_h,\bv_h \big)_{\mathcal{T}_h}
			- \big( p_h, \grad\cdot\bv_h \big)_{\mathcal{T}_h}
			+ \big\langle \widehat{p}(g^{\oplus}(p_h,\bu_h),g^{\ominus}_h), \bn\cdot\bv_h \big\rangle_{\partial\mathcal{T}_h}
			& = 0,
		\end{aligned}
		\right.
	\end{align}
	and
	\begin{align}
		\label{eqn:transvar}
		\langle g^{\ominus}_h-\Pi(g^{\oplus}(p_h,\bu_h)),\xi_h\rangle_{\partial\mathcal{T}_h}=\langle b,\xi_h\rangle_{\partial\mathcal{T}_h},
	\end{align}
	with $g^{\oplus}(p_h,\bu_h)|_{K,F} := p_K+\eta_K\bn_{K,F}\cdot\bu_K$ in the upwind case and $g^{\oplus}(p_h,\bu_h)|_{K,F} := \mathcal{A}^{-1}_F p_K+\bn_{K,F}\cdot\bu_K$ in the symmetric case.
	\label{pbm:chdg}
\end{problem}

To simplify the presentation, we have introduced the \emph{global exchange operator} $\Pi: G_h\rightarrow G_h$ and the \emph{global right-hand side} $b$, whose definitions depend on the choice of numerical fluxes.
For each face $F$ of each element $K$, and for any $g^{\oplus}\in G_h$, they are defined as
\begin{align}
	\Pi(g^{\oplus})|_{K,F} =
	\begin{cases}
		g^{\oplus}_{K',F},
		& \text{if } F\not\subset\partial\Omega, \\
		-g^{\oplus}_{K,F},
		& \text{if } F\subset\Gamma_{\TD},       \\
		g^{\oplus}_{K,F},
		& \text{if } F\subset\Gamma_{\TN},       \\
		0,
		& \text{if } F\subset\Gamma_{\TR},
	\end{cases}
	\quad\text{and}\quad
	b|_{K,F} =
	\begin{cases}
		0,
		& \text{if }F\not\subset\partial\Omega, \\
		2s_{\TD},
		& \text{if } F\subset\Gamma_{\TD},      \\
		-2\eta_Ks_{\TN},
		& \text{if } F\subset\Gamma_{\TN},      \\
		s_{\TR},
		& \text{if } F\subset\Gamma_{\TR},
	\end{cases}
\end{align}
for the upwind fluxes, and as
\begin{align}
	\Pi(g^{\oplus})|_{K,F} =
	\begin{cases}
		g^{\oplus}_{K',F},
		& \text{if } F\not\subset\partial\Omega, \\
		-g^{\oplus}_{K,F},
		& \text{if } F\subset\Gamma_{\TD},       \\
		g^{\oplus}_{K,F},
		& \text{if } F\subset\Gamma_{\TN},       \\
		\mathcal{B}_{K,F,+}^{-1}\mathcal{B}_{K,F,-} g^{\oplus}_{K,F},
		& \text{if } F\subset\Gamma_{\TR},
	\end{cases}
	\quad\text{and}\quad
	b|_{K,F} =
	\begin{cases}
		0,
		& \text{if } F\not\subset\partial\Omega, \\
		2\mathcal{A}_F^{-1}s_{\TD},
		& \text{if } F\subset\Gamma_{\TD},       \\
		-2 s_{\TN},
		& \text{if } F\subset\Gamma_{\TN},       \\
		2\eta_K^{-1}\mathcal{B}_{K,F,+}^{-1}s_{\TR},
		& \text{if } F\subset\Gamma_{\TR},
	\end{cases}
\end{align}
for the symmetric fluxes.


\subsection{Local element-wise discrete problems}

The physical fields $p_h$ and $\bu_h$ are eliminated by solving local element-wise problems where the incoming transmission variable $g_h^{\ominus}$ is considered as a given datum.
In the symmetric case, for each element $K$, the local problem reads:
\begin{problem}
	\label{pbm:loc}
	Find $(p_K,\bu_K)\in \mathcal{P}_\mathrm{p}(K)\times\boldsymbol{\mathcal{P}}_\mathrm{p}(K)$ such that, for all $(q_K,\bv_K)\in\mathcal{P}_\mathrm{p}(K)\times\boldsymbol{\mathcal{P}}_\mathrm{p}(K)$,
	\begin{align}
		\left\{
		\begin{aligned}
			&
			- \imath\big(\kappa_K\eta_K^{-1}p_K,q_K\big)_K
			- \big(\bu_K,\grad q_K\big)_K
			+ \sum_{F\in\mathcal{F}_K}\frac{1}{2}\big\langle(\mathcal{A}_F^{-1}p_K+\bn_{K,F}\cdot\bu_K),q_K\big\rangle_{F}
			\\
			& \quad\quad\quad
			= \sum_{F\in\mathcal{F}_K}\frac{1}{2}\big\langle g^{\ominus}_{K,F},q_K\big\rangle_{F},
			\\
			&
			- \imath\big(\kappa_K\eta_K\bu_K,\bv_K\big)_K
			- \big(p_K,\grad\cdot\bv_K\big)_K
			+ \sum_{F\in\mathcal{F}_K}\frac{1}{2}\big\langle(p_K+\mathcal{A}_F(\bn_{K,F}\cdot\bu_K)),\bn_{K,F}\cdot\bv_K\big\rangle_{F}
			\\
			& \quad\quad\quad
			= - \sum_{F\in\mathcal{F}_K}\frac{1}{2}\big\langle \mathcal{A}_Fg^{\ominus}_{K,F},\bn_{K,F}\cdot\bv_K\big\rangle_{F},
		\end{aligned}
		\right.
	\end{align}
	for a given surface datum $g^{\ominus}_{K,F}\in\mathcal{P}_\mathrm{p}(F)$ for all $F\in\mathcal{F}_K$.
\end{problem}
This local discrete problem corresponds to a discretized Helmholtz problem defined on $K$ with a non-homogeneous Robin boundary condition on $\partial K$.
This problem is well-posed without any condition, as stated and proved hereafter. In contrast to the standard HDG framework, these local problems are also well-posed at the continuous level.
\begin{theorem} [Well-posedness of the local discrete problem]
	\label{thr:well-posedness}
	Problem \ref{pbm:loc} is well-posed.
\end{theorem}
\begin{proof}
	We simply have to prove that, if $g^{\ominus}_{K,F}=0$ for all $F\in\mathcal{F}_K$, the unique solution of Problem \ref{pbm:loc} is $p_K=0$ and $\textbf{u}_K=\textbf{0}$.
	For brevity, the subscripts $K$ and $F$ are omitted for the local fields, the test functions, the outgoing unit normal and the coefficients.
	Taking both equations of Problem \ref{pbm:loc} with $q=p$ and $\bv=\bu$ gives
	\begin{align}
		- \imath(\kappa\eta^{-1}p,p)_K
		- (\bu,\grad p)_K
		+ \sum_{F\in\mathcal{F}_K}\frac{1}{2}\langle(\mathcal{A}^{-1}p+\bn\cdot\bu),p\rangle_{F}
		& = 0,
		\\
		- \imath(\kappa\eta\bu,\bu)_K
		- (p,\grad\cdot\bu)_K
		+ \sum_{F\in\mathcal{F}_K}\frac{1}{2}\langle(p+\mathcal{A}(\bn\cdot\bu)),\bn\cdot\bu\rangle_{F}
		& = 0.
	\end{align}
	Integrating by parts in both equations and taking the complex conjugate lead to
	\begin{align}
		\imath(\kappa\eta^{-1}p,p)_K
		+ (p,\grad\cdot\bu)_K
		+ \sum_{F\in\mathcal{F}_K}\frac{1}{2}\langle p,(\mathcal{A}^{-1}p-\bn\cdot\bu)\rangle_{F}
		& = 0,
		\\
		\imath(\kappa\eta\bu,\bu)_K
		+ (\bu,\grad p)_K
		+ \sum_{F\in\mathcal{F}_K}\frac{1}{2}\langle \bn\cdot\bu,(-p+\mathcal{A}(\bn\cdot\bu))\rangle_{F}
		& = 0.
	\end{align}
	Adding the four previous equations yields
	\begin{align}
		\sum_{F\in\mathcal{F}_K}\left(\frac{1}{2}\langle\mathcal{A}^{-1}p,p\rangle_{F}+\frac{1}{2}\langle\mathcal{A}(\bn\cdot\bu),\bn\cdot\bu\rangle_{F}+\frac{1}{2}\langle p,\mathcal{A}^{-1}p\rangle_{F}+\frac{1}{2}\langle\bn\cdot\bu,\mathcal{A}(\bn\cdot\bu)\rangle_{F}\right) = 0.
	\end{align}
	Since $\mathcal{A}$ and $\mathcal{A}^{-1}$ are self-adjoint, we have
	\begin{align}
		\langle \mathcal{A}^{-1}p,p\rangle_{\partial K}+\langle \mathcal{A}(\bn\cdot\bu),\bn\cdot\bu\rangle_{\partial K} = 0,
	\end{align}
	which gives $p=0$ and $\bn\cdot\bu=0$ on $\partial K$, thanks to the positivity of $\mathcal{A}$ and $\mathcal{A}^{-1}$.
	By using these boundary conditions in Problem \ref{pbm:loc}, we have that the fields should be a solution of the strong problem.
	Because there is no solution with both homogeneous Neumann and Dirichlet boundary conditions, this yields the result.
\end{proof}

\begin{remark}[Case with upwind fluxes]
	It is possible to write a local element-wise discrete problem similar to Problem \ref{pbm:loc} in the case of upwind fluxes and prove its well-posedness by following the steps of the proof of Theorem \ref{thr:well-posedness}.
	The task can be achieved with a similar reasoning.
	For brevity, the definition of the problem and the proof of its well-posedness are omitted.
\end{remark}

\subsection{Abstract form of the hybridized system}
The hybridized CHDG problem can be written in a convenient abstract form by introducing the \emph{global scattering operator} $\TS:G_h\rightarrow G_h$ defined such that, for each face $F$ of each element $K$,
\begin{align}
	\TS(g_h^{\ominus})|_{K,F} := \mathcal{A}_F^{-1}p_K(g_h^{\ominus})+\bn_{K,F}\cdot\bu_K(g_h^{\ominus}),
	\label{scat_oper}
\end{align}
where $(p_K(g_h^{\ominus}),\bu_K(g_h^{\ominus}))$ is the solution of Problem \ref{pbm:loc} with the incoming transmission variable $(g_{K,F}^{\ominus})_{F\in\mathcal{F}_K}$ contained in $g_h^{\ominus}$ as a given surface datum.
This operator can be interpreted as an ``\emph{incoming transmission variable to outgoing transmission variable}'' operator.

By using the operator $\TS$, Problem \ref{pbm:loc} is rewritten as:
\begin{problem}
	\label{pbm:red1}
	Find $g_h^{\ominus}\in G_h$ such that, for all $\xi_h\in G_h$,
	\begin{align}
		\big\langle g_h^{\ominus},\xi_h\big\rangle_{\partial\mathcal{T}_h}-\big\langle \Pi(\TS(g_h^{\ominus})),\xi_h\big\rangle_{\partial\mathcal{T}_h}=\big\langle b,\xi_h\big\rangle_{\partial\mathcal{T}_h}.
	\end{align}
\end{problem}
We introduce the \emph{global projected right-hand side} $b_h := P_hb\in G_h$, where $P_h: L^2(\partial\mathcal{T}_h)\rightarrow G_h$ is the projection operator defined such that $\big\langle P_hb,\xi_h\big\rangle_{\partial\mathcal{T}_h}=\big\langle b,\xi_h\big\rangle_{\partial\mathcal{T}_h}$ for all $\xi_h\in G_h$.
Problem \ref{pbm:red1} can then be rewritten as:
\begin{problem}
	\label{pbm:red2}
	Find $g_h^{\ominus}\in G_h$ such that
	\begin{align}
		(\TI-\Pi\TS)g_h^{\ominus}=b_h,
	\end{align}
\end{problem}
where $\TI$ is the identity operator on $G_h$.
Problem \ref{pbm:red1} and Problem \ref{pbm:red2} are equivalent to Problem \ref{pbm:chdg} because the element-wise local problems are well-posed.


\subsection{Strict contraction of the operator {$\Pi\TS$} \emph{(case with symmetric fluxes)}}

An interesting property of the CHDG formulation when using symmetric fluxes satisfying Assumption \ref{assumption} is that the operator $\Pi\TS$ is a strict contraction.
As a consequence, Problem \ref{pbm:red2} is always well-posed, and it can be solved with the fixed-point iteration without relaxation.

Properties of $\TS$ and $\Pi$ are proved by using norms associated to $G_h$ and $\bigoplus_{F\in\mathcal{F}_K}\mathcal{P}_\mathrm{p}(F)$ defined as
\begin{align}
	\|g_h^{\ominus}\| := \sqrt{\sum_{K\in\mathcal{T}_h}\sum_{F\in\mathcal{F}_K}\|g_{K,F}^{\ominus}\|^2_F}
	\qquad\text{and}\qquad
	\|u\|_{\partial K} := \sqrt{\sum_{F\in\mathcal{F}_K}\|u\|_F^2},
\end{align}
where $\|\cdot\|_F$ is the norm of $L^2(F)$ induced by $\mathcal{A}_F$, namely
\begin{align}
	\|u\|_F := \sqrt{ \big\langle\mathcal{A}_F u,u \big\rangle_F}.
\end{align}

\begin{lemma}
	(i) The solution of Problem \ref{pbm:loc} verifies
	\begin{multline}
		\sum_{F\in\mathcal{F}_K}\|\mathcal{A}_F^{-1}p_K+\bn_{K,F}\cdot\bu_K\|^2_F
		+\sum_{F\in\mathcal{F}_K}\|\mathcal{A}_F^{-1}p_K-\bn_{K,F}\cdot\bu_K-g^{\ominus}_{K,F}\|^2_F= \sum_{F\in\mathcal{F}_K}\|g^{\ominus}_{K,F}\|^2_F.\label{eqn:lem2}
	\end{multline}
	(ii) The second term on the left-hand side of \eqref{eqn:lem2} vanishes if and only if $g^{\ominus}_{K,F}=0$.
	\label{well-pos-loc2}
\end{lemma}
\begin{proof}
	For brevity, the subscripts $K$ and $F$ are omitted for the local fields, the test functions, the unit outgoing normal, the surface data and the coefficients.
	
	\emph{(i)} Taking both equations of Problem \ref{pbm:loc} with $q=p$ and $\bv=\bu$ gives
	\begin{align}
		-\imath(\kappa\eta^{-1}p,p)_K
		-(\bu,\grad p)_K
		+ \frac{1}{2} \big\langle \mathcal{A}(\mathcal{A}^{-1}p+\bn\cdot\bu),\mathcal{A}^{-1}p\big\rangle_{\partial K}
		&
		= \frac{1}{2} \big\langle \mathcal{A}g^{\ominus},\mathcal{A}^{-1}p\big\rangle_{\partial K},
		\\
		-\imath(\kappa\eta\bu,\bu)_K
		-(p,\grad\cdot\bu)_K
		+ \frac{1}{2} \big\langle \mathcal{A}(\mathcal{A}^{-1}p+\bn\cdot\bu),\bn\cdot\bu\big\rangle_{\partial K}
		&
		=- \frac{1}{2} \big\langle \mathcal{A}g^{\ominus},\bn\cdot\bu\big\rangle_{\partial K}.
	\end{align}
	Integrating by parts in both equations and taking the complex conjugate lead to
	\begin{align}
		\imath(\kappa\eta^{-1}p,p)_K
		+ (p,\div\bu)_K
		+ \frac{1}{2} \big\langle \mathcal{A}^{-1}p,\mathcal{A}(\mathcal{A}^{-1}p-\bn\cdot\bu)\big\rangle_{\partial K}
		&
		= \frac{1}{2} \big\langle \mathcal{A}^{-1}p,\mathcal{A}g^{\ominus}\big\rangle_{\partial K},
		\\
		\imath(\kappa\eta\bu,\bu)_K
		+ (\bu,\grad p)_K
		- \frac{1}{2} \big\langle \bn\cdot\bu,\mathcal{A}(\mathcal{A}^{-1}p-\bn\cdot\bu)\big\rangle_{\partial K}
		&
		= - \frac{1}{2} \big\langle \bn\cdot\bu,\mathcal{A}g^{\ominus}\big\rangle_{\partial K}.
	\end{align}
	Adding the four previous equations and multiplying by two yield
	\begin{multline}
		\big\langle \mathcal{A}(\mathcal{A}^{-1}p+\bn\cdot\bu),\mathcal{A}^{-1}p+\bn\cdot\bu\big\rangle_{\partial K}
		+ \big\langle \mathcal{A}^{-1}p-\bn\cdot\bu,\mathcal{A}(\mathcal{A}^{-1}p-\bn\cdot\bu)\big\rangle_{\partial K}
		\\
		= \big\langle \mathcal{A}g^{\ominus},\mathcal{A}^{-1}p-\bn\cdot\bu\big\rangle_{\partial K}
		+ \big\langle (\mathcal{A}^{-1}p-\bn\cdot\bu),\mathcal{A}g^{\ominus}\big\rangle_{\partial K}.
		\label{eqn:proof:intermResu}
	\end{multline}
	Using the following identities for rewriting the right-hand side,
	\begin{align}
		&
		\big\langle \mathcal{A}g^{\ominus},\mathcal{A}^{-1}p-\bn\cdot\bu\big\rangle_{\partial K}
		+ \big\langle (\mathcal{A}^{-1}p-\bn\cdot\bu),\mathcal{A}g^{\ominus}\big\rangle_{\partial K}
		\\
		& \qquad\qquad
		=
		\big\langle \mathcal{A}(\mathcal{A}^{-1}p-\bn\cdot\bu),\mathcal{A}^{-1}p-\bn\cdot\bu\big\rangle_{\partial K}
		+ \big\langle \mathcal{A}g^{\ominus},g^{\ominus}\big\rangle_{\partial K}
		\\
		& \qquad\qquad\qquad\qquad
		- \big\langle \mathcal{A}(\mathcal{A}^{-1}p-\bn\cdot\bu-g^{\ominus}),\mathcal{A}^{-1}p-\bn\cdot\bu-g^{\ominus}\big\rangle_{\partial K}
		\\
		& \qquad\qquad
		= \|\mathcal{A}^{-1}p-\bn\cdot\bu\|^2_{\partial K}
		- \|\mathcal{A}^{-1}p-\bn\cdot\bu-g^{\ominus}\|^2_{\partial K}
		+ \|g^{\ominus}\|^2_{\partial K},
	\end{align}
	equation \eqref{eqn:proof:intermResu} becomes
	\begin{multline}
		\|\mathcal{A}^{-1}p+\bn\cdot\bu\|^2_{\partial K}
		+ \|\mathcal{A}^{-1}p-\bn\cdot\bu\|^2_{\partial K}
		\\
		= \|\mathcal{A}^{-1}p-\bn\cdot\bu\|^2_{\partial K}
		- \|\mathcal{A}^{-1}p-\bn\cdot\bu-g^{\ominus}\|^2_{\partial K}
		+ \|g^{\ominus}\|^2_{\partial K},
	\end{multline}
	which gives the result (\ref{eqn:lem2}).
	
	\emph{(ii)} If the second term on the left-hand side of \eqref{eqn:lem2} vanishes, then $g^{\ominus}=\mathcal{A}^{-1}p-\bn\cdot\bu$ on $\partial K$.
	Using this relation in Problem \ref{pbm:loc}, we see that $p$ and $\bu$ must satisfy
	\begin{align}
		-\i \big( \kappa\eta^{-1} p,q \big)_K
		- \big( \bu,\grad q \big)_K
		+ \big\langle \bn\cdot\bu,q \big\rangle_{\partial K}
		& = 0, \\
		-\i \big( \kappa\eta \bu,\bv \big)_K
		- \big(p,\grad\cdot\bv\big)_K
		+ \big\langle p,\bn\cdot\bv \big\rangle_{\partial K}
		& = 0,
	\end{align}
	for all $q\in\mathcal{P}_\mathrm{p}(K)$ and $\bv\in\boldsymbol{\mathcal{P}}_\mathrm{p}(K)$, and integration by parts shows that $p$ and $\bu$ solve the Helmholtz equation in strong form.
	Because there is no non-zero polynomial solution to the previous equations, meaning that $p=0$ and $\bu=\mathbf{0}$, and then $g^{\ominus}=0$, this yields the result.
	The converse statement is direct, because the local problem is well-posed.
\end{proof}

\begin{theorem}
	\label{thm:opS:cont}
	The scattering operator $\TS$ is a strict contraction, i.e.
	\begin{align}
		\|\TS(g_h^{\ominus})\|<\|g_h^{\ominus}\|, \quad \forall g_h^{\ominus}\in G_h\backslash\{0\}.
	\end{align}
\end{theorem}
\begin{proof}
	Let $g_h^{\ominus}\in G_h\backslash\{0\}$.
	By Lemma \ref{well-pos-loc2}, one has
	\begin{align}
		\sum_{F\in\mathcal{F}_K}\|\mathcal{A}_F^{-1}p_K+\bn_{K,F}\cdot\bu_K\|^2_F<\sum_{F\in\mathcal{F}_K}\|g_{K,F}^{\ominus}\|^2_F.
	\end{align}
	Note that the equality cannot happen because $g_h^{\ominus}\ne0$.
	Then, by using the definition of $\TS$ in equation \eqref{scat_oper}, one has
	\begin{align}
		\sum_{F\in\mathcal{F}_K}\|\TS(g_h^{\ominus})|_{K,F}\|^2_F<\sum_{F\in\mathcal{F}_K}\|g_{K,F}^{\ominus}\|^2_F.
	\end{align}
	Summing this estimate over all $K\in\mathcal{T}_h$ gives the result.
\end{proof}

\begin{theorem}
	\label{thm:opP:cont}
	(i) If $\Gamma_{\TR}=\emptyset$, $\Pi$ is an involution, i.e.~$\Pi^2=\TI$, and an isometry, i.e.
	\begin{align}
		\|\Pi(g_h^{\ominus})\|=\|g_h^{\ominus}\|, \quad \forall g_h^{\ominus}\in G_h.
	\end{align}
	(ii) If $\Gamma_{\TR}\not=\emptyset$, the exchange operator $\Pi$ is a contraction, i.e.
	\begin{align}
		\|\Pi(g_h^{\ominus})\|\le\|g_h^{\ominus}\|, \quad \forall g_h^{\ominus}\in G_h.
	\end{align}
	The inequality is strict for all $g_h^{\ominus}\in G_h$ such that there is at least one face $F\subset\Gamma_\TR$ where $g_{K,F}^\ominus$ is non-zero.
\end{theorem}
\begin{proof}
	\emph{(i)} The result is a straightforward consequence of the definition of $\Pi$ for interior faces and for boundary faces belonging to $\Gamma_{\TD}$ or $\Gamma_{\TN}$.
	
	\emph{(ii)}
	For every boundary face $F\subset\Gamma_\TR$ belonging to an element $K$, we have to prove that
	\begin{align}
		\|(1+\eta_K^{-1}\mathcal{A}_F)^{-1}(1-\eta_K^{-1}\mathcal{A}_F) g\|_F \leq \|g\|_F,
		\qquad\forall g\in\mathcal{P}_\mathrm{p}(F).
	\end{align}
	For any given $g\in\mathcal{P}_\mathrm{p}(F)$, the inequality holds if and only if
	\begin{align}
		\|(1-\eta_K^{-1}\mathcal{A}_F) \xi\|_F \leq \|(1+\eta_K^{-1}\mathcal{A}_F) \xi\|_F,
	\end{align}
	with $\xi := (1+\eta_K^{-1}\mathcal{A}_F)^{-1}g$.
	Because
	\begin{align}
		\|(1+\eta_K^{-1}\mathcal{A}_F) \xi\|_F^2
		&
		= \|\xi\|_F^2
		+ \|(\eta_K^{-1}\mathcal{A}_F) \xi\|_F^2
		+ 2\langle\eta_K^{-1}\mathcal{A}_F^2 \xi, \xi\rangle_F,
		\\
		\|(1-\eta_K^{-1}\mathcal{A}_F) \xi\|_F^2
		&
		= \|\xi\|_F^2
		+ \|(\eta_K^{-1}\mathcal{A}_F) \xi\|_F^2
		- 2\langle\eta_K^{-1}\mathcal{A}_F^2 \xi,\xi\rangle_F,
	\end{align}
	and $\langle\eta_K^{-1}\mathcal{A}_F^2 \xi,\xi\rangle_F\geq0$, the result holds true.
	In addition, if $g\neq0$, then $\xi\neq0$ and the inequality is strict.
\end{proof}

\begin{corollary}
	\label{corol}
	The operator $\Pi\TS$ is a strict contraction, i.e.
	\begin{align}
		\|\Pi\TS(g^{\ominus}_h)\|<\|g^{\ominus}_h\|, \quad \forall g^{\ominus}_h\in G_h\backslash\{0\}.
	\end{align}
\end{corollary}
\begin{proof}
	Corollary \ref{corol} is a direct consequence of Theorems \ref{thm:opS:cont} and \ref{thm:opP:cont}.
\end{proof}

\begin{remark}[Interpretation]
	The strict contraction of $\TS$ is a numerical effect: we have used the fact that the Helmholtz equation has no homogeneous polynomial solution in the proof of Lemma~\ref{well-pos-loc2}.
	The version of equation \eqref{eqn:lem2} for the continuous case corresponds to a conservation of energy.
	By contrast, the properties of $\Pi$ are related to the boundary conditions, and would be preserved at the continuous level.
	This operator is contracting only with a boundary condition that does not preserve energy, i.e.~the Robin boundary condition.
	More precisely, it is strictly contracting only for fields that are different from zero on $\Gamma_{\TR}$.
\end{remark}

\section{Numerical comparison of the methods}
\label{numerical_results}

In this section, iterative solution procedures for the HDG and CHDG methods are studied and compared by using two simple benchmarks and a more realistic application.
For the first two benchmarks, we consider configurations with homogeneous and heterogeneous media in Sections \ref{sec:num:homogeneous} and \ref{sec:num:heterogeneous}, respectively.
For the last benchmark (Section \ref{sec:num:marmousi}), the complex heterogeneous medium illustrates a realistic subsurface medium for seismic wave propagation.


\subsection{Benchmark and numerical setting}
\label{sec:num:benchmarks}

The numerical simulations are performed with a dedicated \textsf{MATLAB} code already used in Ref.~\cite{modave2023}.
The mesh generation and the visualization are performed with \textsf{gmsh} \cite{geuzaine2009}.
In all the cases, third-order polynomial Lobatto basis functions ($\mathrm{p}=3$) are used, and $h$ is the element size specified in \textsf{gmsh}.
For the symmetric fluxes, we take $\mu_F=\sqrt{\eta_K\eta_{K'}}$ and $\kappa_F=\sqrt{\kappa_K\kappa_{K'}}$ for an interior face $F$ shared by two elements $K$ and $K'$.

\paragraph{Benchmark 1 (Plane wave).}
The first benchmark represents the reflection of an incident plane wave at the interface between two media.
The problem is defined on a square domain $\Omega=(0,1)\times(0,1)$ partitioned into two rectangular regions $\Omega_1=(0,1/2)\times(0,1)$ and $\Omega_2=(1/2,1)\times(0,1)$ corresponding to the two media.
The wavenumber and impedance, which are constant in each region, are denoted by $\kappa_1$ and $\eta_1$ in $\Omega_1$, and by $\kappa_2$ and $\eta_2$ in $\Omega_2$.
The exact solution in region $\Omega_1$ is the sum of the incident and reflected waves, while in region $\Omega_2$ it is the transmitted wave.
The solution satisfies the continuity of $p$ and $\bn\cdot\bu$ across the interface ($x=1/2$) between the two regions $\Omega_1$ and $\Omega_2$.
It can be written as
\begin{align}
	p_\mathrm{ref}(x,y)=\left\{
	\begin{aligned}
		& e^{\imath\kappa_1(x\cos\theta_I +y\sin\theta_I)}+Re^{\imath\kappa_1(-x\cos\theta_I+y\sin\theta_I)},
		&
		& \text{in $\Omega_1$},
		\\
		& T e^{\imath\kappa_2(x\cos\theta_T+y\sin\theta_T)},
		&
		& \text{in $\Omega_2$},
	\end{aligned}
	\right.
\end{align}
with the reflection coefficient $R$ and transmission coefficient $T$ given by
\begin{align}
	R
	= \frac{\eta_2\cos\theta_I-\eta_1\cos\theta_T}{\eta_1\cos\theta_T+\eta_2\cos\theta_I}
	e^{\imath\kappa_1\cos\theta_I}
	\qquad\text{and}\qquad
	T
	= \frac{2\eta_2\cos\theta_I}{\eta_1\cos\theta_T+\eta_2\cos\theta_I}
	e^{\imath(\kappa_1\cos\theta_I-\kappa_2\cos\theta_T)/2}.
\end{align}
The angle $\theta_I$ of the incident wave is set to $\theta_I=\pi/4$ in the following.
The angle of the transmitted wave is $\theta_T=\arcsin(\frac{\kappa_1}{\kappa_2}\sin\theta_I)$.
The reference solution is shown for cases with homogeneous and heterogeneous media in Figures \ref{fig:snapshot:refsol:A} and \ref{fig:snapshot:refsol:B}, respectively.
In the numerical model, these solutions are enforced by using non-homogeneous Robin boundary conditions on the boundary of $\Omega$ with boundary data specified by the reference solution, i.e.~$s_\TR=p_\mathrm{ref}-\eta\bn\cdot\bu_\mathrm{ref}$ on $\partial\Omega=\Gamma_\TR$.

\begin{figure}[tb!]
	\begin{subfigure}[t]{.24\linewidth}
		\centering
		\caption{Plane wave with \\ homogeneous medium}
		\label{fig:snapshot:refsol:A}
		\smallskip
		\includegraphics[width=0.85\linewidth]{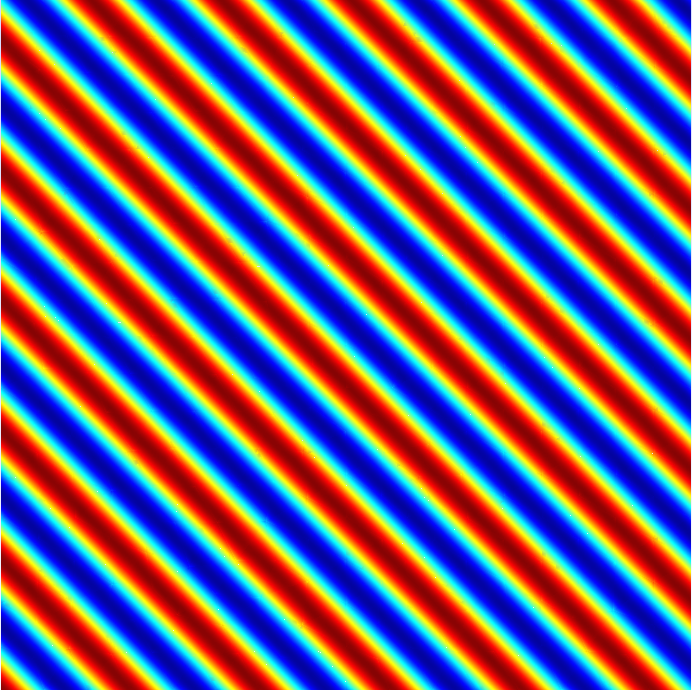}
		\includegraphics[width=0.85\linewidth]{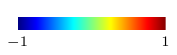}
	\end{subfigure}
	\begin{subfigure}[t]{.24\linewidth}
		\centering
		\caption{Plane wave with \\ heterogeneous medium}
		\label{fig:snapshot:refsol:B}
		\smallskip
		\includegraphics[width=0.85\linewidth]{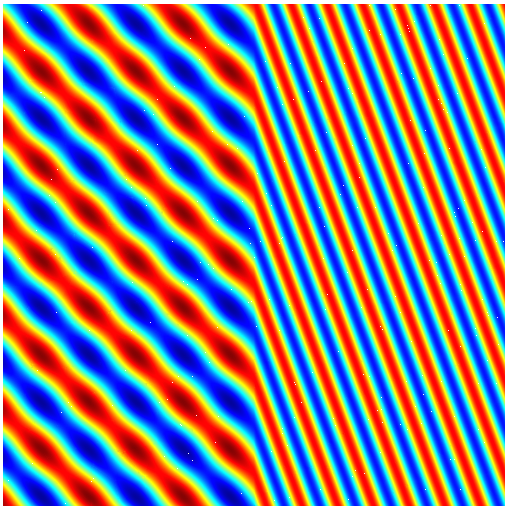}
		\includegraphics[width=0.85\linewidth]{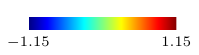}
	\end{subfigure}
	\begin{subfigure}[t]{.24\linewidth}
		\centering
		\caption{Cavity with \\ homogeneous medium}
		\label{fig:snapshot:refsol:C}
		\smallskip
		\includegraphics[width=0.85\linewidth]{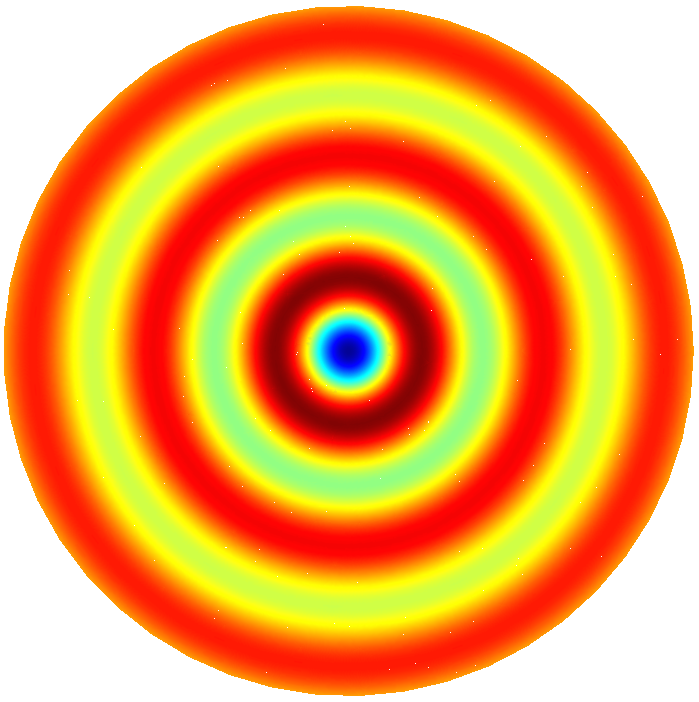}
		\includegraphics[width=0.85\linewidth]{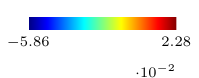}
	\end{subfigure} 
	\begin{subfigure}[t]{.24\linewidth}
		\centering
		\caption{Cavity with \\ heterogeneous medium}
		\label{fig:snapshot:refsol:D}
		\smallskip
		\includegraphics[width=0.85\linewidth]{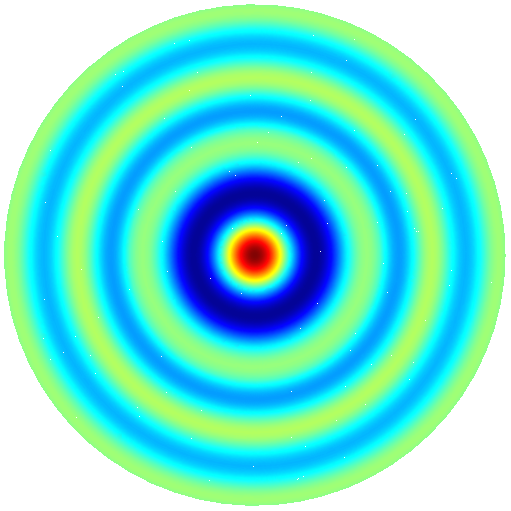}
		\includegraphics[width=0.85\linewidth]{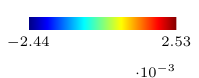}
	\end{subfigure}
	\caption{Real part of the reference solutions $p_{\mathrm{ref}}$ for both benchmark problems with the default parameters corresponding to the first cases of each benchmark in Tables \ref{tab:homog} and \ref{tab:inhomog}.}
	\label{fig:snapshot:refsol}
\end{figure}

\paragraph{Benchmark 2 (Cavity).}
We consider a circular domain $\Omega$ with radius $1/2$ partitioned into the circular region $\Omega_1$ with radius $1/4$ and the annulus $\Omega_2$ with the radial coordinate $r\in(1/4,1/2)$, all centered at the origin.
The physical coefficients are constant in each region, with the same notation as in the first benchmark.
A homogeneous Dirichlet boundary condition is prescribed on the boundary of the domain, i.e.~$\partial\Omega=\Gamma_\TD$ and $s_\TD=0$.
The constant volume source term $f=-1/(\imath\kappa\eta)$ is introduced on the right-hand side of the first equation in \eqref{eqn:pbm}.

If the physical coefficients do not correspond to a resonance mode, the solution of this benchmark is unique and real.
In this case, the reference solution, which depends only on the radial coordinate $r$, is given by
\begin{align}
	p_{\mathrm{ref},j}(r)=A_j \mathrm{J}_0(\kappa_jr)+B_j \mathrm{Y}_0(\kappa_j r)-\kappa_j^{-2}, \quad \text{in $\Omega_j$}, \quad j=1,2,
\end{align}
where $A_j$ and $B_j$ are constant coefficients, $\mathrm{J}_0$ is the Bessel function of the first kind of order 0 and $\mathrm{Y}_0$ is the Bessel function of the second kind of order 0.
To preclude singular solutions at the origin, one has to set $B_1=0$.
The other three coefficients can be determined by enforcing the continuity of $p$ and $\bn\cdot\bu$ across the interface and the homogeneous Dirichlet boundary condition.
This leads to the following linear system:
\begin{align}
	\begin{bmatrix}
		0                                & \mathrm{J}_0(\kappa_2R_2)         & \mathrm{Y}_0(\kappa_2R_2)         \\
		\mathrm{J}_0(\kappa_1R_1)        & -\mathrm{J}_0(\kappa_2R_1)        & -\mathrm{Y}_0(\kappa_2R_1)        \\
		\mathrm{J}_1(\kappa_1R_1)/\eta_1 & -\mathrm{J}_1(\kappa_2R_1)/\eta_2 & -\mathrm{Y}_1(\kappa_2R_1)/\eta_2
	\end{bmatrix}
	\begin{bmatrix}
		A_1 \\
		A_2 \\
		B_2
	\end{bmatrix}=
	\begin{bmatrix}
		\kappa_2^{-2}               \\
		\kappa_1^{-2}-\kappa_2^{-2} \\
		0
	\end{bmatrix},
\end{align}
with $R_1=1/4$ and $R_2=1/2$. Resonance phenomena occur whenever the values of $\kappa_1 R_1$ and $\kappa_2 R_2$ are such that the determinant of the matrix associated to the linear system is zero, namely the matrix is not invertible.
The reference solutions are shown for cases with homogeneous and heterogeneous media in Figures \ref{fig:snapshot:refsol:C} and \ref{fig:snapshot:refsol:D}, respectively.

\paragraph{Iterative solvers.}
Three standard iterations are considered: the fixed-point iteration (only for the CHDG system), the CGNR \emph{(conjugate gradient normal residual)} iteration, and the GMRES \emph{(generalized minimal residual)} iteration without restart.
For a given system $\matalg{A}\vecalg{g}=\vecalg{b}$, the second one corresponds to the conjugate gradient iteration applied to the normal system $\vecalg{A}^*\vecalg{Ag}=\vecalg{A}^*\vecalg{b}$, see e.g.~\cite{saad2003, quarteroni2013}.

Similarly to the approach used in \cite{modave2023}, we have used a symmetric preconditioning with the mass matrix $\matalg{M}$ associated to the faces of the elements.
Denoting the Cholesky factorization of the mass matrix with $\matalg{M} = \matalg{L}\matalg{L}^\top$, this leads to the system $\tilde{\matalg{A}}\tilde{\vecalg{g}}=\tilde{\vecalg{b}}$ with $\tilde{\matalg{A}}:=\matalg{L}^{-1}\matalg{A}\matalg{L}^{-\top}$, $\tilde{\vecalg{g}}:=\matalg{L}^{\top}\vecalg{g}$ and $\tilde{\vecalg{b}}:=\matalg{L}^{-1}\vecalg{b}$.
The preconditioned system corresponds to the one which would be used if orthonormal basis functions would be used on the face, see \cite{modave2023}.
Then, the $2$-norm of an algebraic vector $\tilde{\vecalg{g}}$ is equal to the $L^2$-norm of the corresponding field, i.e.~$\|\tilde{\vecalg{g}}\|_2 = \sqrt{\tilde{\vecalg{g}}^*\tilde{\vecalg{g}}} = \sqrt{\vecalg{g}^*\matalg{M}\vecalg{g}} = \|g_h\|$.

\paragraph{Numerical error.}
In this study, we consider a relative numerical error based on the energy norm of the physical fields.
It is defined as
\begin{equation}
	\text{relative error} := \frac{\|p_h-p_\mathrm{ref},\bu_h-\bu_\mathrm{ref}\|_E}{\|p_\mathrm{ref},\bu_\mathrm{ref}\|_E}
	\;,\quad\text{with }
	\|p,\bu\|^2_E := \sum_{K\in\mathcal{T}_h} \frac{\|p_K\|^2_{L^2(K)}}{2\rho_K c_K^2} + \tfrac{1}{2}\rho_K \|\bu_K\|^2_{L^2(K)}
	\;,
\end{equation}
where $(p_h,\bu_h)$ is the numerical solution, $(p_\mathrm{ref},\bu_\mathrm{ref})$ is the reference solution, and $\|.\|_E$ denotes the energy norm.


\subsection{Comparison for homogeneous media}
\label{sec:num:homogeneous}

The benchmark problems are solved for constant physical parameters corresponding to two levels of difficulty: low- and high-frequency cases for the plane-wave benchmark, and wavenumbers close and very close to the resonance mode $\kappa_3\approx17.3075$ for the cavity benchmark.
For the circular cavity with homogeneous medium, a resonance corresponds to $\kappa_j=2x_j$, where $\{x_j\}_{j\in\mathbb{N}}$ are the zeros of $\mathrm{J}_0(x)$.
In all the cases, the mesh sizes are chosen to achieve a relative error close to $10^{-2}$ when a direct solver is employed.
The parameters are given in Table~\ref{tab:homog}.

\begin{table}[tbh!]
	\centering
	\begin{subtable}{\textwidth}
		\centering
		\caption{Benchmark 1 (plane wave)}
		\smallskip
		\small
		\begin{tabular}{|c|c|c|c|c|c|c|c|c|c|}
			\hline
			Case
			& $\kappa$
			& $h$
			& Numerical flux
			& $\rho(\matalg{\Pi S})$
			& Relative error
			\\
			\hline
			\multirow{2}*{1}
			& \multirow{2}*{$15\pi$}
			& \multirow{2}*{$1/16$}
			& Sym-0
			& $1 - 4.18\cdot10^{-3}$ 
			& $1.44\cdot10^{-2}$
			\\
			\cline{4-6}
			&
			&
			& Sym-2
			& $1 - 7.97\cdot10^{-3}$ 
			& $1.70\cdot10^{-2}$
			\\
			\hline
			\multirow{2}*{2}
			& \multirow{2}*{$30\pi$}
			& \multirow{2}*{$1/34$}
			& Sym-0
			& $1 - 1.88\cdot10^{-3}$ 
			& $1.37\cdot10^{-2}$
			\\
			\cline{4-6}
			&
			&
			& Sym-2
			& $1 - 3.48\cdot10^{-3}$ 
			& $1.59\cdot10^{-2}$
			\\
			\hline
		\end{tabular}
		\label{tab_open_hom}
	\end{subtable}
	\medskip
	\begin{subtable}{\textwidth}
		\centering
		\caption{Benchmark 2 (cavity)}
		\smallskip
		\small
		\begin{tabular}{|c|c|c|c|c|c|c|c|c|c|}
			\hline
			Case
			& $\kappa$
			& $h$
			& Numerical flux
			& $\rho(\matalg{\Pi S})$
			& Relative error
			\\
			\hline
			\multirow{2}*{1}
			& \multirow{2}*{$16.5$}
			& \multirow{2}*{$0.04$}
			& Sym-0
			& $1 - 1.22\cdot10^{-8}$ 
			& $1.07\cdot10^{-2}$
			\\
			\cline{4-6}
			&
			&
			& Sym-2
			& $1 - 1.71\cdot10^{-6}$ 
			& $1.07\cdot10^{-2}$
			\\
			\hline
			\multirow{2}*{2}
			& \multirow{2}*{$17$}
			& \multirow{2}*{$0.025$}
			& Sym-0
			& $1 - 1.15\cdot10^{-9}$ 
			& $1.16\cdot10^{-2}$
			\\
			\cline{4-6}
			&
			&
			& Sym-2
			& $1 - 1.44\cdot10^{-7}$ 
			& $1.16\cdot10^{-2}$
			\\
			\hline
		\end{tabular}
		\label{tab_cav_hom}
	\end{subtable}
	\caption{Parameters for the benchmark problems with homogeneous media, including the spectral radius $\rho(\matalg{\Pi S})$ of the matrix of the CHDG hybridized system, and the relative numerical error. In these cases, $\omega=\kappa$, $c=1$, $\rho=1$ and $\eta=1$.}
	\label{tab:homog}
\end{table}

\begin{figure}[!p]
	\centering
	\includegraphics[width=1\linewidth]{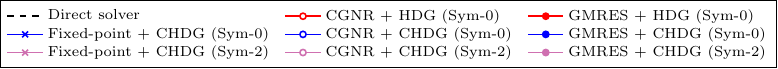}
	\\
	\medskip
	\begin{subfigure}[b]{0.49\textwidth}
		\centering
		\caption{Benchmark 1 -- Case 1 ($\kappa=15\pi$)}
		\includegraphics[width=1\linewidth]{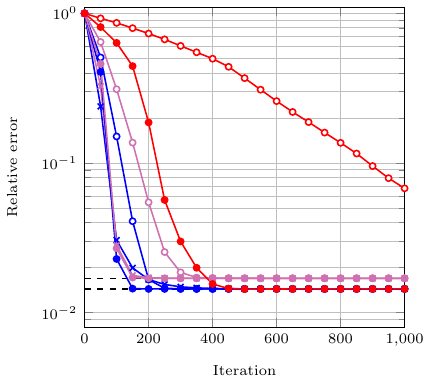}
	\end{subfigure}
	\hfill
	\begin{subfigure}[b]{0.49\textwidth}
		\centering
		\caption{Benchmark 1 -- Case 2 ($\kappa=30\pi$)}
		\includegraphics[width=1\linewidth]{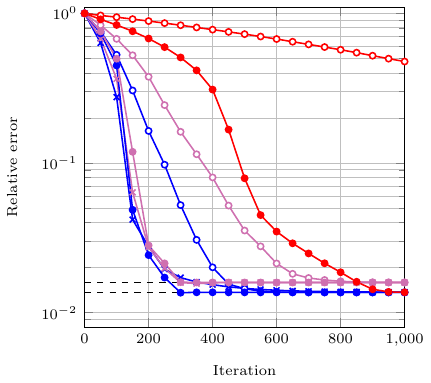}
	\end{subfigure}
	\\
	\medskip
	\begin{subfigure}[b]{0.49\textwidth}
		\centering
		\caption{Benchmark 2 -- Case 1 ($\kappa=16.5$)}
		\includegraphics[width=1\linewidth]{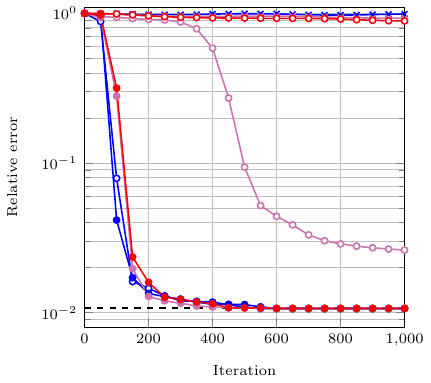}
	\end{subfigure}
	\hfill
	\begin{subfigure}[b]{0.49\textwidth}
		\centering
		\caption{Benchmark 2 -- Case 2 ($\kappa=17$)}
		\includegraphics[width=1\linewidth]{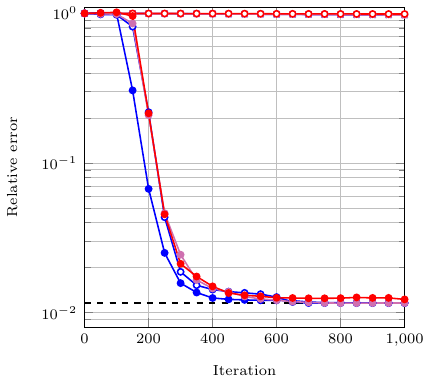}
	\end{subfigure}
	\caption{Results for the benchmark problems with homogeneous media.
		Error history with different iterative schemes and different DG methods.
		The dashed horizontal lines correspond to the relative numerical errors obtained with a direct solver for the different numerical fluxes.}
	\label{fig:histo:homog}
\end{figure}

The history of relative error is plotted in Figure \ref{fig:histo:homog} during the fixed-point iterations (lines with marker $\times$, only for CHDG), the CGNR iterations (lines with marker $\circ$) and the GMRES iterations (lines with marker $\bullet$) applied to the HDG and CHDG hybridized systems.
The 0th-order symmetric flux (Sym-0) is considered for both methods, and the 2nd-order symmetric flux (Sym-2) is used for CHDG.
The relative errors obtained with a direct solver is indicated by the horizontal dashed lines for both numerical fluxes.
Let us note that, since the impedance $\eta$ is constant, the upwind flux is identical to the 0th-order symmetric flux with $\mathcal{A}_F=\eta$.

The following observations can be made:
\begin{itemize}
	\item
	The fixed-point iteration applied to the CHDG system converges in all the cases.
	The convergence is fast for the plane-wave benchmark, and very slow for the cavity benchmark.
	In both cases, the convergence is slower for the parameters corresponding to the more difficult configuration (i.e.~high frequency case or close to a resonance frequency).
	These results are in accordance with the spectral radii given in Table~\ref{tab:homog}: the convergence is slower for radii closer to~$1$.
	\item
	Comparing the numerical fluxes used in the CHDG method, we observe that the convergence of the fixed point iteration is comparable with `Sym-0' and `Sym-2'.
	The results are similar for the GMRES iteration.
	In contrast, the CGNR iteration with `Sym-2' is slower than with `Sym-0', sometimes significantly so.
	\item
	The convergence of the CGNR and GMRES iterations is slower with HDG than with CHDG for the plane-wave benchmark.
	The difference is more important for the high-frequency case.
	For the cavity benchmark, the convergence of GMRES is similar with both methods, while the convergence of CGNR is slower with HDG than with CHDG.
\end{itemize}

In a nutshell, CHDG with `Sym-0' is one of the best solution in all the cases.
The fixed-point iteration applied to the CHDG system converges, but its performance is strongly influenced by the physical parameters, and it becomes unusable for cavities.
The second-order numerical flux does not accelerate the convergence of the iterative procedures.


\subsection{Comparison for heterogeneous media}
\label{sec:num:heterogeneous}

The benchmark problems are now tested with two sets of parameters.
For both benchmarks, the impedance $\eta$ is constant in the first set, it is not constant in the second set, and the wavenumber $\kappa$ is not constant in all the sets.
The parameters are given in Table~\ref{tab:inhomog}.
For all the cases, the mesh size is chosen to obtain a relative error close to $10^{-2}$ when a direct solver is employed.

\begin{table}[tbh!]
	\centering
	\begin{subtable}{\textwidth}
		\centering
		\caption{Benchmark 1 (plane wave)}
		\smallskip
		\small
		\begin{tabular}{|c|c|c|c|c|c|c|c|c|c|}
			\hline
			Case
			& $\omega$
			& $c$
			& $\rho$
			& $\kappa$
			& $\eta$
			& $h$
			& Numerical flux
			& $\rho(\matalg{\Pi S})$
			& Relative error
			\\
			\hline
			\multirow{3}*{1}
			& \multirow{3}*{$15\pi$}
			& \multirow{3}*{\makecell{$1$     \\ $1/2$}}
			& \multirow{3}*{\makecell{$1$     \\ $2$}}
			& \multirow{3}*{\makecell{$15\pi$ \\ $30\pi$}}
			& \multirow{3}*{\makecell{$1$     \\ $1$}}
			& \multirow{3}*{\makecell{$1/16$  \\ $1/34$}}
			& Upw
			& $1 - 1.63\cdot10^{-3}$          
			& $1.01\cdot10^{-2}$
			\\
			\cline{8-10}
			&
			&
			&
			&
			&
			&
			& Sym-0
			& $1 - 1.63\cdot10^{-3}$          
			& $1.01\cdot10^{-2}$
			\\
			\cline{8-10}
			&
			&
			&
			&
			&
			&
			& Sym-2
			& $1 - 4.32\cdot10^{-3}$          
			& $1.20\cdot10^{-2}$
			\\
			\hline
			\multirow{3}*{2}
			& \multirow{3}*{$15\pi$}
			& \multirow{3}*{\makecell{$1$     \\ $1/2$}}
			& \multirow{3}*{\makecell{$1$     \\ $1$}}
			& \multirow{3}*{\makecell{$15\pi$ \\ $30\pi$}}
			& \multirow{3}*{\makecell{$1$     \\ $1/2$}}
			& \multirow{3}*{\makecell{$1/16$  \\ $1/34$}}
			& Upw
			& $1 + 2.38\cdot10^{-2}$          
			& $9.75\cdot10^{-3}$
			\\
			\cline{8-10}
			&
			&
			&
			&
			&
			&
			& Sym-0
			& $1 - 1.55\cdot10^{-2}$          
			& $9.74\cdot10^{-3}$
			\\
			\cline{8-10}
			&
			&
			&
			&
			&
			&
			& Sym-2
			& $1 - 4.31\cdot10^{-3}$          
			& $1.16\cdot10^{-2}$
			\\
			\hline
		\end{tabular}
		\label{tab_open_inhom}
	\end{subtable}
	\medskip
	\begin{subtable}{\textwidth}
		\centering
		\caption{Benchmark 2 (cavity)}
		\smallskip
		\small
		\begin{tabular}{|c|c|c|c|c|c|c|c|c|c|}
			\hline
			Case
			& $\omega$
			& $c$
			& $\rho$
			& $\kappa$
			& $\eta$
			& $h$
			& Numerical flux
			& $\rho(\matalg{\Pi S})$
			& Relative error
			\\
			\hline
			\multirow{3}*{1}
			& \multirow{3}*{$10\pi$}
			& \multirow{3}*{\makecell{$1$     \\ $2/3$}}
			& \multirow{3}*{\makecell{$1$     \\ $3/2$}}
			& \multirow{3}*{\makecell{$10\pi$ \\ $15\pi$}}
			& \multirow{3}*{\makecell{$1$     \\ $1$}}
			& \multirow{3}*{\makecell{$1/12$  \\ $1/16$}}
			& Upw
			& $1 - 5.19\cdot10^{-5}$          
			& $1.57\cdot10^{-2}$
			\\
			\cline{8-10}
			&
			&
			&
			&
			&
			&
			& Sym-0
			& $1 - 5.19\cdot10^{-5}$          
			& $1.57\cdot10^{-2}$
			\\
			\cline{8-10}
			&
			&
			&
			&
			&
			&
			& Sym-2
			& $1 - 5.28\cdot10^{-5}$          
			& $1.67\cdot10^{-2}$
			\\
			\hline
			\multirow{3}*{2}
			& \multirow{3}*{$10\pi$}
			& \multirow{3}*{\makecell{$1$     \\ $2/3$}}
			& \multirow{3}*{\makecell{$1$     \\ $1$}}
			& \multirow{3}*{\makecell{$10\pi$ \\ $15\pi$}}
			& \multirow{3}*{\makecell{$1$     \\ $2/3$}}
			& \multirow{3}*{\makecell{$1/12$  \\ $1/16$}}
			& Upw
			& $1 + 3.45\cdot10^{-4}$          
			& $8.59\cdot10^{-3}$
			\\
			\cline{8-10}
			&
			&
			&
			&
			&
			&
			& Sym-0
			& $1 - 4.99\cdot10^{-5}$          
			& $8.61\cdot10^{-3}$
			\\
			\cline{8-10}
			&
			&
			&
			&
			&
			&
			& Sym-2
			& $1 - 5.29\cdot10^{-5}$          
			& $1.01\cdot10^{-2}$
			\\
			\hline
		\end{tabular}
		\label{tab_cav_inhom}
	\end{subtable}
	\caption{Parameters of the benchmark problems with heterogeneous media, including spectral radius $\rho(\matalg{\Pi S})$ of the matrix of the CHDG hybridized system, and the relative numerical error.}
	\label{tab:inhomog}
\end{table}

\begin{figure}[!p]
	\centering
	\includegraphics[width=1\linewidth]{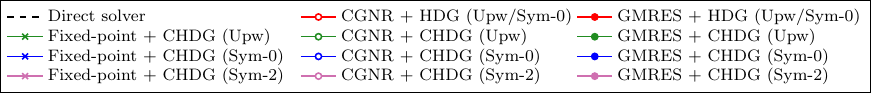}
	\\
	\medskip
	\begin{subfigure}[b]{0.49\textwidth}
		\centering
		\caption{Benchmark 1 -- Case 1 \\ \small ($\kappa_1=15\pi$, $\kappa_2=30\pi$, $\eta_1=1$, $\eta_2=1$)}
		\includegraphics[width=1\linewidth]{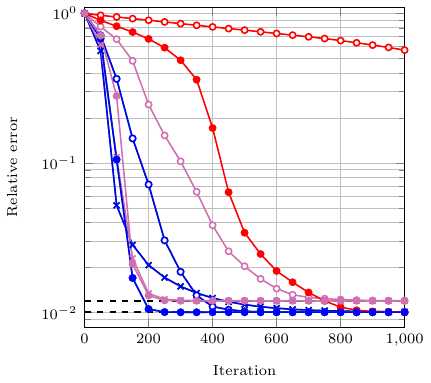}
	\end{subfigure}
	\hfill
	\begin{subfigure}[b]{0.49\textwidth}
		\centering
		\caption{Benchmark 1 -- Case 2 \\ \small ($\kappa_1=15\pi$, $\kappa_2=30\pi$, $\eta_1=1$, $\eta_2=1/2$)}
		\label{fig:histo:heterog:b}
		\includegraphics[width=1\linewidth]{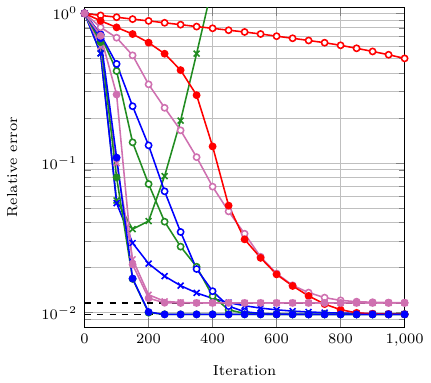}
	\end{subfigure}
	\\
	\medskip
	\begin{subfigure}[b]{0.49\textwidth}
		\caption{Benchmark 2 -- Case 1 \\ \small ($\kappa_1=10\pi$, $\kappa_2=15\pi$, $\eta_1=1$, $\eta_2=1$)}
		\centering
		\includegraphics[width=1\linewidth]{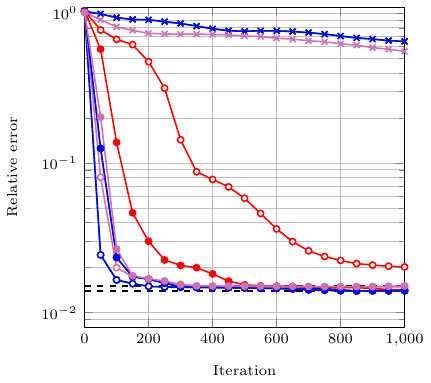}
	\end{subfigure}
	\hfill
	\begin{subfigure}[b]{0.49\textwidth}
		\centering
		\caption{Benchmark 2 -- Case 2 \\ \small ($\kappa_1=10\pi$, $\kappa_2=15\pi$, $\eta_1=1$, $\eta_2=2/3$)}
		\label{fig:histo:heterog:d}
		\includegraphics[width=1\linewidth]{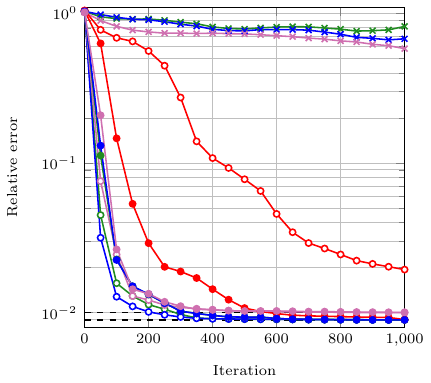}
	\end{subfigure}
	\caption{Results of the benchmark problems with heterogeneous media.
		Error history with different iterative schemes and different DG methods.
		The dashed horizontal lines correspond to the relative numerical errors obtained with a direct solver for the different numerical fluxes.
		In most graphs, the green lines (CHDG `Upw') are hidden by the corresponding blue lines (CHDG `Sym-0').}
	\label{fig:histo:inhomog}
\end{figure}

The history of relative error is plotted in Figure \ref{fig:histo:inhomog} for different combinations of hybridized methods, numerical fluxes and iterative schemes.
For CHDG, we have considered the upwind fluxes (Upw) and both zeroth- and second-order symmetric fluxes (Sym-0 and Sym-2).
The numerical fluxes `Upw' and `Sym-0' are identical when the impedance is constant.
Nonetheless, the results are very similar when the impedance is not constant.
For HDG, both fluxes lead to the same results, even when the impedance is discontinuous.

We observe that the fixed-point iteration applied to the CHDG systems with the symmetric fluxes converges in all cases, while it fails to converge with the upwind fluxes when the impedance is discontinuous.
The divergence can be seen in Figure \ref{fig:histo:heterog:b}, and it also appears when increasing the number of iterations for the case corresponding to Figure \ref{fig:histo:heterog:d}.
These results are consistent with the theory and the spectral radii in Table \ref{tab:inhomog}: the contraction property of the operator $\Pi\TS$ was proved in Corollary \ref{corol} only for symmetric fluxes, and the convergence occurs numerically for radii strictly less than 1.
The convergence is quite fast for the plane-wave benchmark, while it slows down dramatically for the cavity benchmark.

Similarly to the cases with homogeneous media, when comparing the symmetric numerical fluxes used in the CHDG system, we observe that the convergence of CGNR and GMRES with `Sym-2' is always either comparable or slower than with `Sym-0' for both benchmarks.
Once again, the convergence of the iterative schemes is much slower with HDG than with CHDG, regardless of the iterative procedure that is used.

In a nutshell, the CGNR and GMRES iterations combined with CHDG and either `Upw' or `Sym-0' are always the best options.
The fixed-point iteration applied to CHDG with `Sym-0' always converges, whereas it does not with HDG or CHDG with `Upw'.


\subsection{Realistic application}
\label{sec:num:marmousi}

To compare the methods on a more illustrative case, we consider the well-known Marmousi benchmark problem, which consists of a realistic underground structure commonly used to assess numerical methods, see e.g.~\cite{vion2018}.
The variations of the velocity $c(\bx)$ and the density $\rho(\bx)$ are shown in Figure \ref{fig:marmousi}.

An acoustic field is generated by a source point at the coordinates $(4585\mathrm{m},-10\mathrm{m})$, near the top of the rectangular domain $\Omega={]0\mathrm{m},9192\mathrm{m}[}\times{]-2094\mathrm{m},0\mathrm{m}[}$.
The mesh is generated with \textsf{gmsh} \cite{geuzaine2009}, and includes $60\,761$ elements.
The element size is adjusted following the empirical rule $h \approx \lambda(\bx)/n_{\lambda}$, with the wavelength $\lambda(\bx) = c(\bx)/f$, the frequency $f=30\:\mathrm{Hz}$, the angular frequency $\omega=2\pi f$, the mesh density $n_{\lambda} = 10/(\mathrm{p}+1)$, and the polynomial degree $\mathrm{p}=3$.
The physical coefficients are constant in each cell.
The Dirichlet boundary condition $p=0$ is applied at the top of the domain, and the absorbing boundary condition $p-\eta\bn\cdot\bu=0$ is prescribed on the other sides of the domain.
The mesh and the solution are shown in Figure \ref{fig:marmousi}.

\begin{figure}[tb!]
	\centering
	\begin{subfigure}[t]{0.49\linewidth}
		\centering
		\caption{Velocity $c(\bx)$}
		\label{marmousi_2pi30_vel}
		\includegraphics[width=\linewidth]{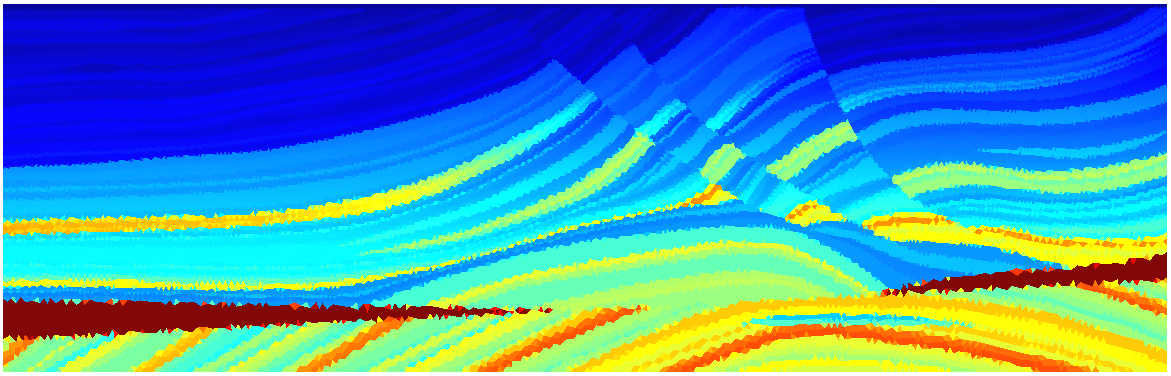}
		\includegraphics[width=0.85\linewidth]{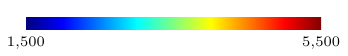}
	\end{subfigure}
	\hfill
	\begin{subfigure}[t]{0.49\linewidth}
		\centering
		\caption{Density $\rho(\bx)$}
		\label{marmousi_2pi30_den}
		\includegraphics[width=\linewidth]{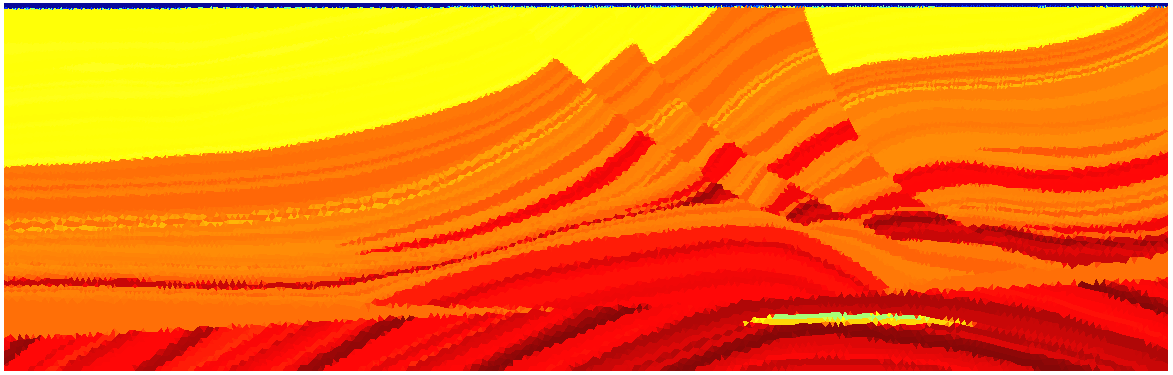}
		\includegraphics[width=0.85\linewidth]{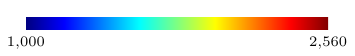}
	\end{subfigure}
	\\
	\begin{subfigure}[t]{0.49\linewidth}
		\centering
		\caption{Mesh}
		\label{marmousi_2pi30_mesh}
		\includegraphics[width=\linewidth]{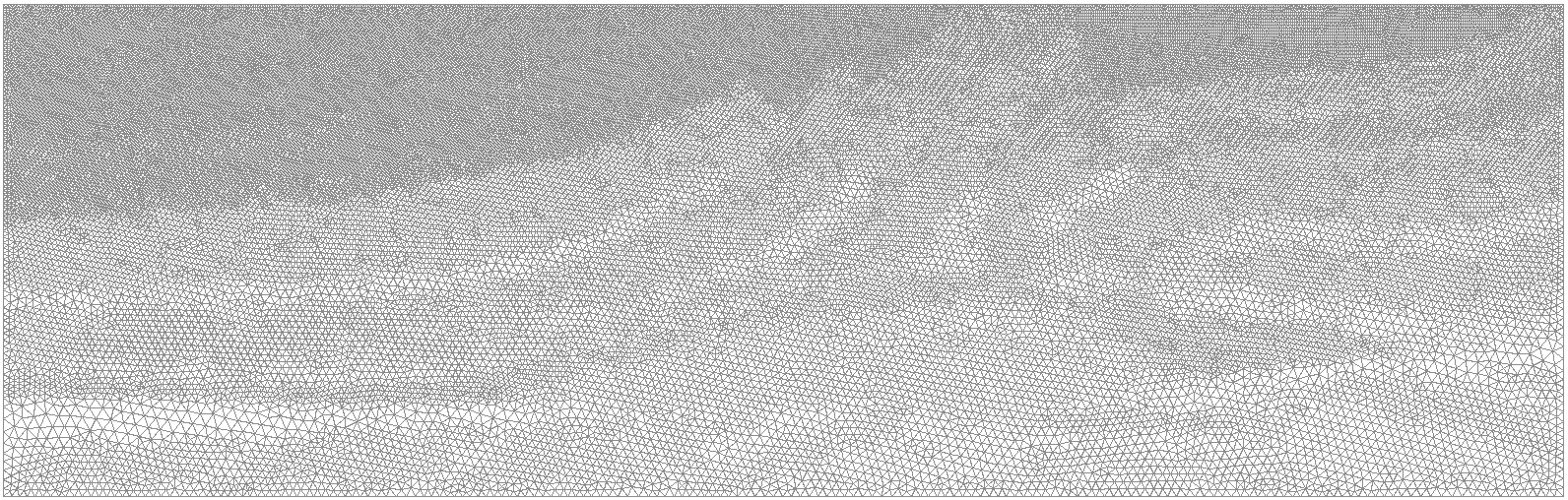}
	\end{subfigure}
	\hfill
	\begin{subfigure}[t]{0.49\linewidth}
		\centering
		\caption{Reference solution}
		\label{marmousi_2pi30_sol}
		\includegraphics[width=\linewidth]{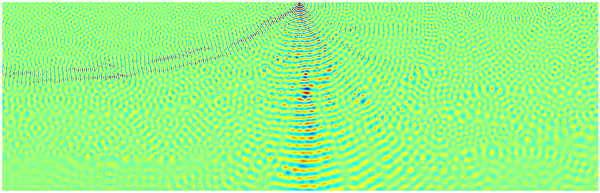}
		\includegraphics[width=0.85\linewidth]{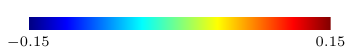}
	\end{subfigure}
	\caption{Marmousi benchmark problem. Profile of velocity $c(\bx)$ (a), of density $\rho(\bx)$(b), mesh (c) and real part of the reference pressure field obtained with a direct solver (d).}
	\label{fig:marmousi}
\end{figure}

\begin{figure}[tb!]
	\centering
	\includegraphics[width=1\linewidth]{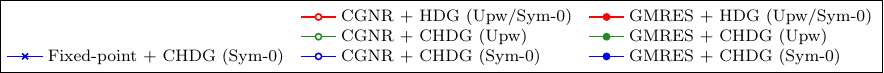}
	\\
	\medskip
	\begin{subfigure}{0.49\textwidth}
		\centering
		\caption{Relative residual}
		\includegraphics[width=1\linewidth]{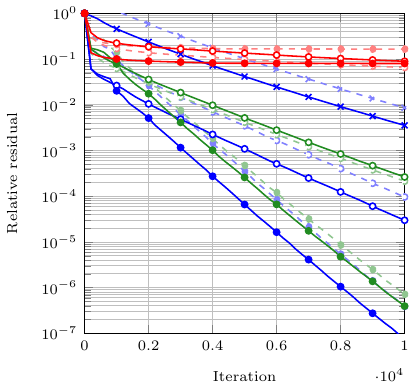}
	\end{subfigure}
	\hfill
	\begin{subfigure}{0.49\textwidth}
		\centering
		\caption{Relative error}
		\includegraphics[width=1\linewidth]{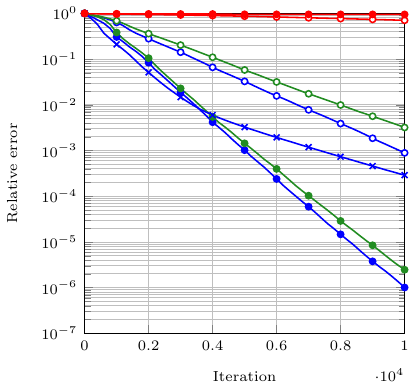}
	\end{subfigure}
	\caption{Marmousi benchmark problem. Relative residual (left) and relative error history (right) with different iterative schemes and DG methods. For the relative residual history, we consider for each case both the residual associated with the hybridized system (in $\times$, plain lines) and the one associated with the physical system (in physical norm, dashed lines). The relative error is computed by comparing the numerical physical solution obtained at each iteration with the reference numerical physical solution obtained with a direct solver.}
	\label{marmousi_history}
\end{figure}

The problem is solved using HDG and CHDG with the zeroth-order symmetric flux (Sym-0) and the upwind flux (Upw).
We consider the fixed-point iteration (only for CHDG with `Sym-0'), the CGNR iteration, and the GMRES iterations (with a restart after every 10 iterations).
The restart strategy, which is widely used in practice, helps to reduce computation and memory usage.

Figure \ref{marmousi_history} shows the histories of residual and error for each combination of methods.
For each case, we show the relative residual associated with the preconditioned hybridized system and the one corresponding to the physical system.
In contrast to the previous sections, the relative error is computed by comparing the numerical solution obtained at each iteration with that obtained with a direct solver.

We observe that the iterative solution of HDG is considerably slower compared to that of CHDG, regardless of whether CGNR or GMRES is used, both in terms of residuals and errors.
The choice of numerical flux has a negligible effect on the convergence of the iterations with HDG (only one curve is shown for clarity), while the iterations with `Upw' are slightly slower than with `Sym-0' when considering CHDG.

Comparing the different iterative strategies for solving the CHDG system with `Sym-0', we observe that the decay of both error and residual is faster with GMRES than with CGNR.
The slopes of the curves of error and residual are nonetheless similar.
The fixed-point iteration performs well until approximately $4\,000$ iterations, where the decay of error slows down.

One should note that these results do not give a final conclusion on which iterative scheme should be used in practice, since the runtime per iteration depends on the iterative scheme.
The fixed-point iteration requires a single multiplication by the matrix of the system, while CGNR involves the multiplications by both the matrix and its adjoint at each iteration.
The cost of GMRES increases at each iteration until the restart, and it is, overall, the most expensive method per iteration.
Because these results are obtained with a \textsf{MATLAB} code that is not fully optimized, the observed runtimes are not reported here as they do not provide a representative comparison of efficiency.
A performance analysis will be performed on a 3D parallel code under development.

\section{Conclusion}
\label{conclusion}

In this work, we have extended and studied the CHDG method, originally proposed for Helmholtz problems in \cite{modave2023}, in order to address wave propagation in heterogeneous media with piecewise constant physical coefficients.
In particular, we considered formulations with standard upwind numerical fluxes, which are widely used in the literature, and symmetric numerical fluxes involving a general impedance operator $\mathcal{A}$.
For both numerical fluxes, transmission variables are introduced at the interfaces between the elements in the hybridization procedure.
In contrast to \cite{modave2023}, the definition of the transmission variables involves the physical coefficients and the operator $\mathcal{A}$.

The CHDG hybridized system can be written in the form $(\TI-\Pi\TS)g=b$.
In the case of symmetric fluxes, if $\mathcal{A}$ is a positive and self-adjoint operator, it is proven that the operator $\Pi\TS$ is a strict contraction, which makes it possible to solve the system with a fixed-point iteration.
This result also holds for the standard upwind fluxes if the impedance $\eta$ is constant.
Otherwise, $\Pi\TS$ is not a contraction and the fixed-point iteration does not converge, as confirmed by the numerical results.

We have systematically studied and compared the different methods, numerical fluxes and iterative schemes by considering 2D benchmark problems.
Most of the observations that were made in the homogeneous case in \cite{modave2023} remain valid.
The fixed-point iteration applied to the CHDG system converges for the cases where it is proven that the operator $\Pi\TS$ is a strict contraction.
While the convergence of the fixed-point iteration is in general very fast for benchmarks with physical dissipation, it is very slow for cavity cases.
Considering the CGNR and GMRES iterations, we observe that their convergence is nearly always faster with CHDG than with the standard HDG method.
The difference is very significant for the benchmarks with impedance boundary conditions, and it is drastic for the Marmousi benchmark.

In this work, we have also investigated the use of symmetric fluxes with a second-order differential operator $\mathcal{A}$, inspired by recent work on domain decomposition methods.
While this operator satisfies the assumptions of the theoretical framework, it always results in a slower convergence of the iterative schemes compared to the symmetric fluxes with a scalar operator.
Comparing the upwind fluxes and the scalar symmetric fluxes, we have observed that, in the cases where they are not identical (i.e.~when the impedance $\eta$ is discontinuous at the interfaces), they provide similar results, except for the fixed-point iteration.

This approach can be extended to 3D cases and to other physical problems, such as aeroacoustic, electromagnetic \cite{rappaport2025} and elastic wave problems.
We are currently making effort in these directions.
We are investigating in more detail the computational aspects, as well as combinations with preconditioning and domain decomposition methods, to further accelerate the convergence of the iterations.

\paragraph{Acknowledgments.}
This work was supported in part by the ANR JCJC project \emph{WavesDG} (research grant ANR-21-CE46-0010).

\small
\setlength{\bibsep}{0pt plus 0ex}
\bibliographystyle{abbrvnat}
\bibliography{myrefs}
\addcontentsline{toc}{section}{References}

\end{document}